\documentclass[10pt,a4paper,reqno]{amsart}               

\usepackage{amsmath,bbm}
\usepackage{amssymb}
\usepackage{graphicx}
\usepackage{color}
\usepackage{latexsym}
\usepackage{multirow}

\usepackage[english]{babel}
\usepackage[utf8]{inputenc}
\usepackage[T1]{fontenc}
\usepackage{cite}

\usepackage{stmaryrd,comment}
\usepackage{url}
\usepackage[plainpages=false,pdfpagelabels,colorlinks=true,citecolor=blue,hypertexnames=false]{hyperref}

\newcommand{\ns}{{\mathbb N}} 
\newcommand{\zs}{{\mathbb Z}} 
\newcommand{\qs}{{\mathbb Q}}  
\newcommand{\cs}{{\mathbb C}} 
\newcommand{\rs}{{\mathbb R}} 

\newcommand{\eps}{\varepsilon}
\newcommand{\vareps}{\varepsilon}
\newcommand{\cR}{\mathcal{R}}

\newcommand{\bC}{C}
\newcommand{\bD}{D}
\newcommand{\bN}{N}
\newcommand{\hP}{\widehat{P}}
\newcommand{\hQ}{\widehat{Q}}

\newcommand{\Q}{T}

\newcommand{\GK}{\mathbb{K}}

\DeclareMathOperator{\Rat}{Rat}
\DeclareMathOperator{\Pol}{Pol}

\DeclareMathOperator{\df}{drf}
\DeclareMathOperator{\dv}{drv}

\DeclareMathOperator{\vv}{v}
\DeclareMathOperator{\ff}{f}
\DeclareMathOperator{\ee}{e}

\DeclareMathOperator{\Tch}{\mathcal T}
\DeclareMathOperator{\Ppol}{P}

\DeclareMathOperator{\rem}{rem}

\newcommand{\LandauO}{\mathcal{O}}

\newcommand{\DT}{\dot{T}_1}
\newcommand{\DK}{\dot{K}_1}

\newcommand{\gQ}{Q}

\newcommand{\cC}{\mathcal C}

\newcommand{\Maple}{{\sc Maple}}
\newtheorem{Theorem}{Theorem}
\newtheorem*{Theorem*}{Theorem~\ref{thm:ED} (repeated)}
\newtheorem{Proposition}[Theorem]{Proposition}

\newtheorem{Corollary}[Theorem]{Corollary}

\newtheorem{Lemma}[Theorem]{Lemma}

\theoremstyle{Definition}

\theoremstyle{remark}
\newtheorem*{Remark}{Remark}

\newcommand{\beq}{\begin{equation}}
\newcommand{\eeq}{\end{equation}}

\newcommand{\gf}{generating function}
\newcommand{\gfs}{generating functions}
\newcommand{\fps}{formal power series}

\def\emm#1,{{\em #1}}

\catcode`\@=11
\def\section{\@startsection{section}{1}%
 \z@{.7\linespacing\@plus\linespacing}{.5\linespacing}%
 {\large\bfseries\scshape\centering}}

\def\subsection{\@startsection{subsection}{2}%
 \z@{.5\linespacing\@plus\linespacing}{.5\linespacing}%
 {\normalfont\bfseries\scshape}}

\def\subsubsection{\@startsection{subsubsection}{3}%
 \z@{.5\linespacing\@plus\linespacing}{-.5em}
 {\normalfont\bfseries\itshape}}
\catcode`\@=12

\begin{document}
\title
[The 3-state Potts model on planar triangulations]
{The 3-state Potts model on planar triangulations:\\
explicit algebraic solution}

\author[M. Bousquet-M\'elou]{Mireille Bousquet-M\'elou}
\address{M. Bousquet-M\'elou: CNRS, LaBRI, Universit\'e de Bordeaux, 
351 cours de la Lib\'eration, 33405 Talence, France}
\email{mireille.bousquet@labri.fr}

\author[H. Notarantonio]{Hadrien Notarantonio}
  \address{H. Notarantonio: CNRS, IRIF, UniversitÃ© Paris CitÃ©, 75013 Paris, France}
  \email{hadrien.notarantonio@irif.fr}

\thanks{MBM was partially supported by the ANR projects DeRerumNatura (ANR-19-CE40-0018), CombinÃ© (ANR-19-CE48-0011), and CartesEtPlus (ANR-23-CE48-0018). HN was partially supported by the ANR projects ISOMA (ANR-21-CE48-0007) and CartesEtPlus (ANR-23-CE48-0018).}


\begin{abstract}
  We consider the  $3$-state Potts generating function $T(\nu,w)$ of planar triangulations; that is, the bivariate series
  that counts planar triangulations with vertices coloured in~$3$ colours, weighted by their size (number of vertices, recorded by the variable $w$) and by the number of monochromatic edges (variable $\nu$).

  This series was proved to  be algebraic  15 years ago by Bernardi and the first author: this follows from its link with the solution of a discrete differential equation (DDE), and from general algebraicity results on such equations. However, despite recent progress on the effective solution of DDEs, the exact value of $T(\nu,w)$ has remained unknown so far --- except in the case  $\nu=0$, corresponding to \emm proper, colourings and  solved by Tutte in the sixties. We determine here this exact value, proving that $T(\nu,w)$ satisfies a polynomial equation of degree~$11$ in $T$ and genus $1$ in~$w$ and $T$. We prove that  the critical value of $\nu$ is $\nu_c=1+3/\sqrt{47}$, with a critical exponent $6/5$ in the series $T(\nu_c, \cdot)$, while the other values of $\nu$ yield the usual map exponent $3/2$.  

  By duality of the planar Potts model, our results also characterize the 3-state Potts generating function of planar \emm cubic, maps, in which all vertices have degree $3$. In particular, the  annihilating polynomial, still of degree $11$, that we obtain for \emm properly, 3-coloured cubic maps proves a conjecture by Bruno Salvy from 2009. 
\end{abstract}

\date{\today}
\maketitle

\section{Introduction and main results}

A  planar map is a connected planar (multi)graph
embedded in the sphere, taken up to orientation preserving
homeomorphism (Figure~\ref{fig:ex-coloured}).  
The enumeration of planar maps is a venerable topic in combinatorics,  born in
the early sixties with the pioneering work of William Tutte~\cite{tutte-triangulations,tutte-census-maps}. Fifteen
years later the topic started a second, independent, life in theoretical
physics, where planar maps provide a discrete model of \emm quantum
gravity,~\cite{BIPZ,BIZ}. The enumeration of maps also has connections with
factorizations of permutations, and hence representations of the
symmetric group~\cite{Jackson:Harer-Zagier,Jackson:character-maps}.
As a result, many techniques have been invented to count families of maps, from the early recursive approaches~\cite{tutte-census-maps} to more and more combinatorial and finally bijective techniques, which rely on a much better understanding of maps~\cite{Sch97,BDG-planaires,bouttier-mobiles,bouttier-guitter-slices,bernardi-fusy-girth}.  Moreover, 40 years after the first
enumerative results of Tutte, planar maps  crossed the border between
combinatorics and probability theory, where they are  studied as
 random metric spaces~\cite{angel-schramm,chassaing-schaeffer,le-gall-topological,marckert-mokkadem}. The limit behaviour of large
planar random maps is now well understood, and gave birth to a
variety of 
 limiting objects, either
 continuous like the Brownian map~\cite{legall,miermont}, or
 discrete  like the UIPQ (uniform infinite planar quadrangulation)~\cite{angel-schramm,chassaing-durhuus,curien-miermont,menard-same}.

The enumeration of maps equipped with some additional structure (a
spanning tree, a proper colouring, a self-avoiding-walk, a configuration of the Ising
model...) has attracted the interest of both combinatorialists and theoretical
physicists since the early days of this
study~\cite{DK88,Ka86,mullin-boisees,lambda12,tutte-dichromatic-sums}. {At
  the moment}, a challenge
is to understand the limiting behaviour of maps equipped with one such
structure~\cite{albenque-Ising,borot-bouttier-duplantier,kassel-wilson,kenyon2015bipolar,richier-perco,sheffield-inventory}.

\begin{figure}[htb]
  \centering
  \includegraphics[width=50mm]{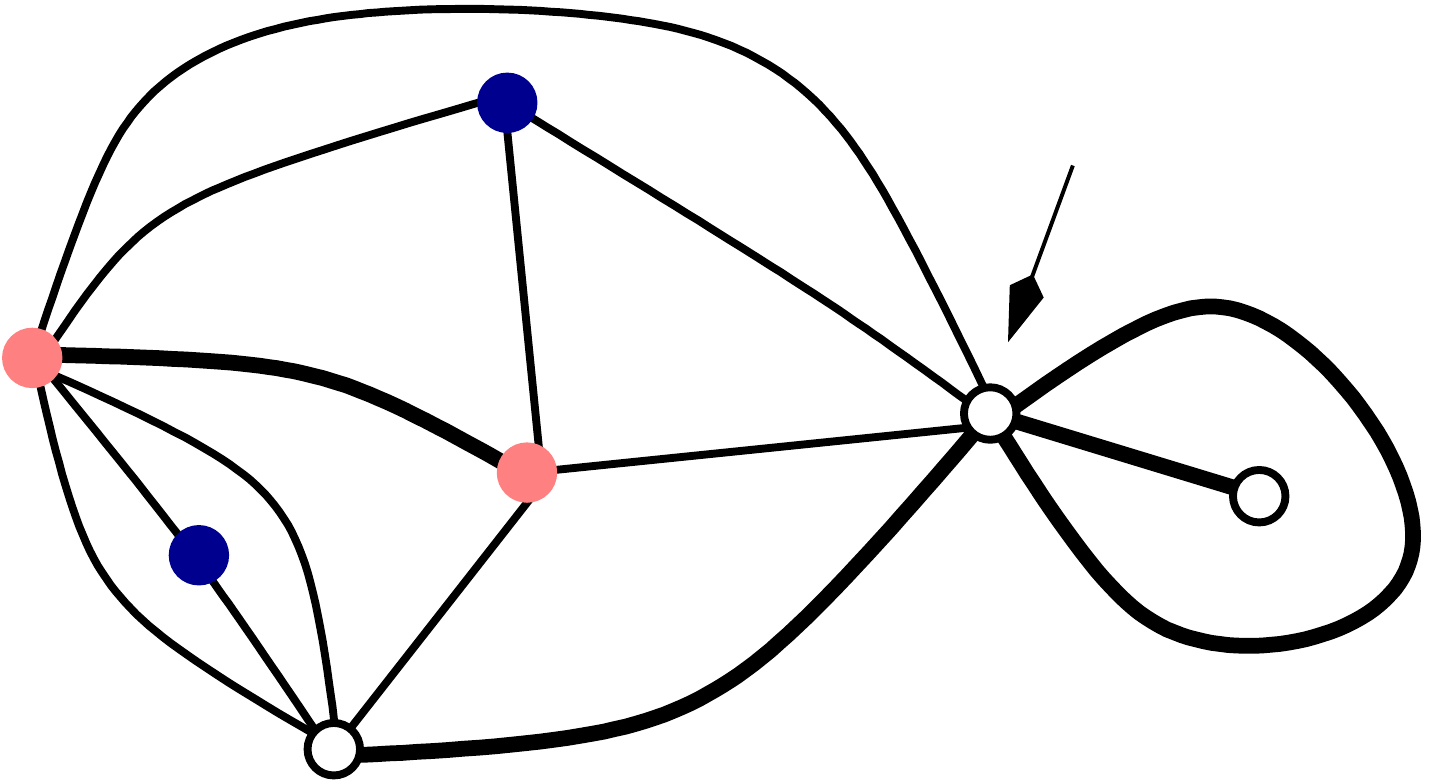}
  \caption{A $3$-coloured rooted near-triangulation having $7$ vertices, $4$ monochromatic edges (in thick lines) and  outer degree $4$. It contributes $w^7 \nu^4y^4$ in the series $T(\nu,w;y)$ counting such maps by vertices ($w$), monochromatic edges~($\nu$) and outer degree ($y$).}
  \label{fig:ex-coloured}
\end{figure}

\medskip
\paragraph{\bf The Potts model.}
An especially interesting structure is a configuration of the \emm $q$-state Potts model,. Given a graph~$G$, the \emm partition function, of this model counts all colourings of the vertices of $G$ with colours taken in $\{1, 2, \ldots, q\}$, with a weight $\nu$ for each \emm monochromatic, edge (that is, an edge having both endpoints of the same colour); see Figure~\ref{fig:ex-coloured}. This partition function is in fact a polynomial in $\nu$ and $q$, henceforth called \emm Potts polynomial, of $G$. Up to a change of variables, it is equivalent to the \emm Tutte polynomial, of $G$. It admits many interesting specializations, like the number of spanning trees or spanning forests of $G$. Of course, when $\nu=0$, one recovers the \emm chromatic polynomial, of $G$, counting all \emm proper, $q$-colourings (with no monochromatic edges). We refer to~\cite{welsh-book} for details.

Studying a family of maps equipped with this model means summing these partition functions over all maps of fixed size in the family. In combinatorial terms, this boils down to counting
$q$-coloured maps by their size and the number of monochromatic edges (and possibly additional statistics). The corresponding \gf\ is the \emm Potts \gf, of the family of maps under study.  This question has already been considered for several families of planar maps, both in physics papers~\cite{daul,eynard-bonnet-potts,zinn-justin-dilute-potts,borot3,guionnet-jones}, and in combinatorics papers~\cite{tutte-dichromatic-sums,BeBM-11,BeBM-17}. For general planar maps first, a natural starting point,  already known to Tutte~\cite{tutte-dichromatic-sums}, is a combinatorially founded functional equation that characterizes the Potts \gf, but requires to introduce two additional statistics on maps, and the corresponding variables  in the \gf; these statistics and variables are sometimes called \emm catalytic,. A similar equation  was established for planar triangulations in~\cite{BeBM-11}, and is recalled in~\eqref{eq:Q} below. Based on these equations, Bernardi and the first author proved in~\cite{BeBM-11,BeBM-17} that the Potts \gfs\ of general planar maps and of planar triangulations both satisfy a polynomial differential equation in the size variable. This equation depends polynomially on $q$ and $\nu$.

This long proof was directly inspired by Tutte's enumeration of triangulations weighted by their chromatic polynomial, which took him $10$ years and $10$ papers; see~\cite{tutte-chromatic-revisited} for a survey. An important, and remarkable, intermediate step establishes that for $q\neq 0,4$ of the form  $q=4\cos^2(k\pi/m)$ (with $k$ and $m$ two integers), one can also write an equation involving a \emm single,  catalytic variable, both for  general maps and for triangulations; this equation is reported in~\eqref{eqinv} for triangulations in the case $q=3$. Such equations are much better understood than those with two catalytic variables~\cite{popescu,mbm-jehanne}, and it was proved in~\cite{BeBM-11} that for these values of $q$,  the Potts \gf\ had to be \emm algebraic,, that is, to satisfy a polynomial equation in $\nu$ and the size variable.

\medskip
\paragraph{\bf Three states.} These special values of $q$ include $q=2$ (the Ising model), and $q=3$, which is the setting of the present paper. For $q=2$, the minimal polynomial of the Potts/Ising \gf\ was derived from the 1-catalytic equation in~\cite{BeBM-11}, both for general maps and for triangulations. But the case $q=3$ resisted, except in the special case $\nu=0$. And it has resisted until this date, despite significant progress on the effective solution of $1$-catalytic equations~\cite{BCNS_ISSAC22,BNS_ISSAC23,Notar-DDE,notarantonio-yurkevich}.  All effective techniques lead to a polynomial system from which one must extract a single polynomial equation satisfied by the main \gf, but all systems that were obtained so far for this problem turned out to be too big to be solved.

In this paper, we use yet another system, for triangulations, which we manage to solve. The maps that we consider are \emm rooted, by choosing a corner at a vertex (Figure~\ref{fig:ex-coloured}). This is a classical choice, which prevents symmetries. The face containing this corner is the  \emm root face,, or \emm outer face,.
Beyond triangulations, we consider \emm near-triangulations,, in which the root face  has any degree, while all other faces are triangles (Figure~\ref{fig:ex-coloured}; precise definitions are given in Section~\ref{sec:defs}). In a companion paper~\cite{BMN-26}, we address the case of general planar maps, where an analogous system exists. Its size, however, is bigger, and the solution technique is different.

The main theorem of the present paper reads as follows.

\begin{Theorem}\label{thm:main}
  For each $i\ge 1$, the \gf\ $T_i\equiv T_i(\nu,w)$ that counts $3$-coloured near-triangulations of outer degree $i$ by vertices (variable $w$) and monochromatic edges (variable $\nu$) is algebraic of degree $11$. All  series $T_i$ belong to the same extension of degree $11$ of $\qs(\nu,w)$.
\end{Theorem}

The smallest  minimal polynomial that we obtain is for the series $\DT:= \partial_w T_1$. This polynomial, given below, has degree $2$ in $w$, $7$ in $\nu$, and of course $11$ in $\DT$.  The genus of the underlying  curve in $w$ and $\DT$ is $1$, so the degree $2$ in $w$ is optimal (the genus would be $0$ if there was an equation of degree $1$ in $w$). This minimal polynomial contains ``only'' $117$ monomials in $\nu,w$ and $\DT$:
\begin{small}
  \allowdisplaybreaks
\begin{multline} \label{alg:dT1}
  276480 \DT^{11} \nu^{7}-27648 \nu^{6} \left(31 \nu +24\right) \DT^{10}+1152 \nu^{5} \left(1021 \nu^{2}+1678 \nu +541\right) \DT^{9}\\
-18 \nu^{4} \left(46080 \nu^{3} w +51935 \nu^{3}+138243 \nu^{2}+92253 \nu +17089\right)  \DT^{8}\\
+72 \nu^3 \left( 1920 \nu^{3} \left(17 \nu +7\right) w +6545 \nu^{4}+25755 \nu^{3}+26863 \nu^{2}+10253 \nu +1144\right) \DT^{7}
\\
-4\nu^2\left(1008 \nu^{3} \left(727 \nu^{2}+586 \nu +127\right) w +38596 \nu^{5}+219355 \nu^{4}+322318 \nu^{3}+190022 \nu^{2}+43274 \nu +2915 \right)
 \DT^{6}
\\
+4\nu \left(216 \nu^{3} \left(2433 \nu^{3}+2879 \nu^{2}+1255 \nu +153\right) w +8027 \nu^{6}+67626 \nu^{5}+134820 \nu^{4}+109109 \nu^{3}  +38007 \nu^{2} \right.\\
\left.+5103 \nu +188 \right)
\DT^{5}
+\left(41472 \nu^{6} \left(\nu -1\right) w^{2}-12 \nu^{3} \left(78871 \nu^{4}+122456 \nu^{3}+80010 \nu^{2}+19688 \nu +1375\right) w \right.
\\
\left. -3876 \nu^{7}-53138 \nu^{6}-145202 \nu^{5}-151460 \nu^{4}-71656 \nu^{3}-14332 \nu^{2}-958 \nu -18\right) \DT^{4}
+\left(13824 \nu^{5} \left(5 \nu +1\right) \left(1-\nu \right) w^{2}
\right.
\\
\left. +8 \nu^{2} \left(5 \nu +1\right) \left(6823 \nu^{4}+11843 \nu^{3}+9045 \nu^{2}+2429 \nu +100\right) w +208 \nu^{7}+6088 \nu^{6}+24600 \nu^{5}+31836 \nu^{4}+19256 \nu^{3} \right.\\
\left. +5040 \nu^{2}+440 \nu +12\right) \DT^{3}
+\left(1728 \nu^{4} \left(\nu -1\right) \left(5 \nu +1\right)^{2} w^{2}-12 \nu  \left(3 \nu +1\right) \left(1358 \nu^{5}+2771 \nu^{4}+2504 \nu^{3}+868 \nu^{2}
  \right. \right.
\\
\left. \left.+58 \nu +1\right) w -312 \nu^{6}-2401 \nu^{5}-3747 \nu^{4}-2821 \nu^{3} 
 -899 \nu^{2}-78 \nu -2\right) \DT^{2}
+\nu \left(-96 \nu^{2} \left(\nu -1\right) \left(5 \nu +1\right)^{3} w^{2}
\right. \\
\left. +4  \left(\nu +1\right) \left(1229 \nu^{5}+2390 \nu^{4}+2114 \nu^{3}+697 \nu^{2}+49 \nu +1\right) w + \left(104 \nu^{4}+189 \nu^{3}+177 \nu^{2}+67 \nu +3\right)\right) \DT\\
+2 \nu^{2} \left(\nu -1\right) \left(5 \nu +1\right)^{4} w^{2}-2 \nu^{2} \left(\nu +2\right) \left(104 \nu^{4}+189 \nu^{3}+177 \nu^{2}+67 \nu +3\right) w =0.
\end{multline}
\end{small}
For comparison, the minimal polynomial of $T_1$ itself contains $1304$ monomials. What led us to compute the minimal polynomial of $\DT$, having first obtained that of $T_1$, is that maps with a marked atom (here, a marked vertex) often yield simpler equations than unmarked maps, because marked maps are often in bijection with trees. This simplification is already clear in  the case $\nu=1$: the minimal polynomial of $T_1$ is
  \[
    192 T_1^{3}-\left(288 w -1\right) T_1^{2}+w \left(90 w -1\right) T_1 -3 w^{3} \left(81 w -1\right)=0,
  \]
  while that of $\DT$ reduces to
  \[
    2 \DT^{3}-3 \DT^{2}+\DT -6 w=0.
  \]

\begin{Remark}
  A \gf\ of degree $11$ in map enumeration is neither astonishing, nor the sign that it cannot be understood combinatorially. For instance, the Ising model on \emm cubic, maps (all vertices have degree $3$) is known to be described by a series of degree $11$, provided one records separately the numbers of black and white vertices; see~\cite[Prop.~7.2]{BMS-arxiv} or~\cite[Chap.~8]{eynard-book}. A related univariate example is provided by the \emm hard particle model, on cubic maps: vertices are colored in black and white, and two vertices colored black cannot be neighbours. One simply records the number of vertices.  Despite the fact that the associated series has degree $11$, this problem is well understood combinatorially, in terms of coloured trees governed by the following system~\cite[Prop.~6.3]{BMS-arxiv}:
  \beq\label{param-HP3}
\left\{
\begin{array}{lll}
P&=&\displaystyle 1+2x^2P^2+\frac{2xPQ(1+xQ)(1+2xQ)}{1-4x^2P},\\
\\
Q&=&\displaystyle  \frac{2xP}{1-4x^2P}+ \frac{xQ^2(1+xQ)^2}{(1-4x^2  P)^3}.
\end{array}
\right.
\eeq
If we eliminate $Q$, we obtain for $P$ an equation of genus $8$. We hope that a combinatorial explanation for the results of this paper will be found as well.
\end{Remark}

\medskip
\paragraph{\bf Special values of $\boldsymbol{\nu}$.} When $\nu=0$, we have $T_1=0$ (since there is a loop at the root vertex), but for $i>1$ we recover known results on the enumeration of properly $3$-coloured near-triangulations of outer degree $i$, with a minimal polynomial of degree only $2$ for $T_i$. In particular, $T_2=T_1/\nu$ for general $\nu$, hence  the above equation implies that, at $\nu=0$,
\[
2 \dot{T}_2^{2}-\left(4 w +3\right) \dot{T}_2 +2 w \left(w +6\right)=0
  .
\]
The first result of this type was obtained  by Tutte already in 1963, formulated in terms of the number of \emm bicubic, (bipartite and cubic) maps~\cite[p.~269]{tutte-census-maps}. Indeed, the duals of bicubic maps are the \emm Eulerian, triangulations (those in which every vertex has even degree). These are the only triangulations that admit a proper $3$-colouring, and each of them admits exactly $6$ colourings.

The case $\nu=1$ is also simple and well-known, as $T_i$ then counts near-triangulations by vertices (with a weight $3w$  per vertex); it has degree $3$ over $\qs(w)$~\cite{mullin-nemeth-schellenberg}. In particular, the minimal polynomial of $\DT$ factors for $\nu=1$, and the factor that vanishes gives
\[
  2 \DT^{3}-3 \DT^{2}+\DT -6 w=0.
\]
The genus is $0$ in these two cases.

\medskip
\paragraph{\bf Duality.} A duality property of the (planar) Potts polynomial, recalled in Section~\ref{sec:q-Potts}, allows us to translate our results in terms of the $3$-Potts \gf\
of \emm near-cubic, maps (those in which all vertices have degree $3$, except possibly the root vertex); see~\eqref{eq:KT} and Corollary~\ref{cor:cubic}. In particular, we prove for \emm properly, $3$-coloured near-cubic maps the following result, equivalent to a conjecture of Bruno Salvy that dates back to 2009 (see~\cite[Conj.~27]{BeBM-11}).

\begin{Corollary}\label{cor:proper_cubic}
  The \gf\ $K_1=4 w^{3}+84 w^{4}+1872 w^{5}+\LandauO( w^{6})$ of properly $3$-coloured near-cubic maps with root degree~$1$, counted by faces, is algebraic of degree $11$. Its derivative satisfies
  \begin{multline*}
 324 w^{2}=    655360 \DK^{11}+1245184 \DK^{10}+866304 \DK^{9}-80 \left(8192 w -1995\right) \DK^{8}-2880 \left(512 w +49\right) \DK^{7}\\
    -504 \left(2944 w +219\right) \DK^{6}-24 \left(36640 w +1383\right) \DK^{5}-\left(16384 w^{2}+334416 w +3033\right) \DK^{4}\\
    -6 \left(4096 w^{2}+13584 w -153\right) \DK^{3}-9 \left(1536 w^{2}+1300 w -33\right) \DK^{2}-27 \left(4 w +1\right) \left(32 w -1\right) \DK.
  \end{multline*}
\end{Corollary}

\paragraph{\bf Singularities and asymptotics.} We also study the singularities of the series $T_i$, as functions of $w$ depending on the parameter $\nu\ge 0$. We find for every $\nu$ the usual asymptotic behaviour of uncoloured planar maps, namely $[w^n] T_i \sim \rho_\nu^{-n} n^{-5/2}$, except at the critical point $\nu_c=1+3/\sqrt{47}$, where $[w^n] T_i \sim  \rho_\nu^{-n} n^{-11/5}$ (up to a multiplicative constant in both cases). This study requires some care, and occupies a significant part of the paper.

\bigskip
\noindent
{\bf Outline of the paper.} We begin in Section~\ref{sec:prelim} with definitions on maps and the Potts model in~$q$ colours. We introduce the \emm Potts \gf,\ $T(y)\equiv T(q,\nu,w;y)$ of near-triangulations, which counts these maps, equipped with a $q$-colouring of their vertices,  by the size (variable~$w$), the number of monochromatic edges (variable $\nu$) and the outer degree ($y$). In Section~\ref{sec:eqfunc} we give three characterizations of~$T(y)$, which were established successively in earlier papers~\cite{BeBM-11,BeBM-17}. The first one  is valid for any $q$, but it involves an additional variable $x$ and a series that is more general than $T(y)$. The other two characterizations  only hold for $q=3$. One involves  $T(y)$, while the other does not involve the variable $y$ any more. It is a polynomial system relating several series in $w$ and $\nu$,  among which $T_1, T_3, T_5$ and $T_7$. In Section~\ref{sec:Ty}, we derive from this system the minimal polynomials of $T_1, T_3, T_5$ and $T_7$, and prove Theorem~\ref{thm:main}. Section~\ref{sec:asympt} is devoted to the singular analysis of the series $T_i$, culminating in Proposition~\ref{prop:asympt}.

Many parts of this paper  require using a computer algebra system.  The paper is thus accompanied by a \Maple\ session available on the web pages of the authors. We also use {\tt msolve}, a C library for solving in arbitrary precision multivariate polynomial systems~\cite{msolve}.

\section{Preliminaries}
\label{sec:prelim}

\subsection{Planar maps}\label{sec:defs}
%
A \emph{planar map} is a proper
embedding of a connected planar graph in the
oriented sphere, considered up to orientation preserving
homeomorphism. Loops and multiple edges are allowed
(Figure~\ref{fig:example-map}, left). The \emph{faces} of a map are the
connected components of  its complement. The numbers of
vertices, edges and faces of a planar map $M$, denoted by $\vv(M)$,
$\ee(M)$ and $\ff(M)$,  are related by Euler's relation
$\vv(M)+\ff(M)=\ee(M)+2$.
The \emph{degree} of a vertex or face is the number
of edges incident to it, counted with multiplicity. A \emph{corner} is
a sector delimited by two consecutive edges around a vertex; 
hence  a vertex or face of degree $k$ is incident to  $k$ corners. 

For counting purposes it is convenient to consider \emm rooted, maps. 
A map is rooted by choosing a corner. The incident vertex and face are called root vertex and root face.
The edge that follows the root corner in counterclockwise 
order around the root vertex is called the root edge.
In figures, we  usually choose the root face
as the infinite face (Figure~\ref{fig:example-map}). 
This explains why we often call the root face the \emm outer face,,
and its degree the \emm outer degree, (denoted $\df(M)$).

The \emph{dual} of a
map $M$, denoted $M^*$, is the map obtained by placing a 
vertex of $M^*$ in each face of $M$ and an edge of $M^*$ across each
edge of $M$; see Figure~\ref{fig:example-map}, right.  The dual of a rooted map is rooted canonically at the dual corner: that is, the root face (resp. root vertex) of $M^*$ is dual to the root vertex (resp. root face) of $M$. A \emm triangulation, is a map in which all faces have degree~$3$. A \emph{near-triangulation} is a map in which all non-root faces have degree~$3$. We denote by $\mathcal T_i$ the set of near-triangulations of outer degree $i$. Duality transforms
triangulations (resp. near-triangulations)   into \emm cubic, (resp. \emm near-cubic,) maps, that is, maps in which every vertex (resp. every non-root vertex) has degree~3.  We denote by $\mathcal K_i$ the set of near-cubic maps of root degree $i$.  Observe that in a near-triangulation~$M$ with outer degree $i$, a double counting of edges gives
\[
  2\ee(M)= i+ 3(\ff(M)-1).
\]
Combined with Euler's relation, this yields
\beq\label{fe-v}
  \ff(M)= 2\vv(M) -i-1, \qquad \ee(M)=3 \vv(M)-i-3 .
\eeq

\begin{figure}[h]
  \centering
  \includegraphics{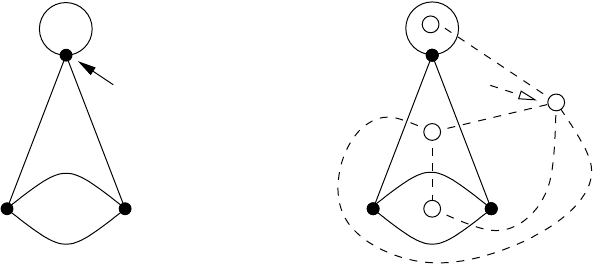}
  \caption{\emph{Left:} a rooted planar map with $3$ vertices and $4$ faces, having outer degree $4$.
    \emph{Right:} the dual map, in dashed edges.}
  \label{fig:example-map}
\end{figure} 

From now on, every map
is \emph{planar} and \emph{rooted}, and these adjectives will
often be omitted. 
We include among rooted planar maps the \emph{atomic map}
having one vertex and no edge.

\subsection{The $\boldsymbol q$-state Potts model}
\label{sec:q-Potts}
For $q\in \ns:=\{1, 2, \ldots\}$, the \emm Potts polynomial, of a map $M$ is defined to be
\[
  \Ppol_M(q, \nu):= \sum_{c  : V(M)\rightarrow \{1, \ldots, q\}}
  \nu^{m(c)},
\]
where $c$ is a colouring of the vertices of $M$ in $q$ colours taken in $\{1, 2, \ldots, q\}$, and $m(c)$ is the number of \emm monochromatic, edges (whose endpoints share the same colour). 
For instance, the map $M$ shown on the left of Figure~\ref{fig:example-map} has Potts polynomial:
\beq\label{Pott-exemple}
  \Ppol_M(q, \nu)= q\nu \left( (q-1)(q-2)+ (q-1)\nu^2+ 2(q-1)\nu + \nu^4 \right).
\eeq
It is easy to prove that $\Ppol_M(q, \nu)$ is not only a polynomial in $\nu$, but also a polynomial in $q$~\cite{welsh-book}.

We define the  \emm Potts \gf\ of near-triangulations, by:
\beq\label{T-def}
T(y)\equiv T(q,\nu,w;y)
:=\frac 1 q \sum_{M} \Ppol_M(q,\nu) w^{\vv(M)} y^{\df(M)},
\eeq
where the sum runs over all planar near-triangulations $M$ (including the atomic map).
Since there are finitely many near-triangulations with a given number of vertices, 
and $\Ppol_M(q,\nu)$ is a multiple of $q$,
the series $T(y)$ is a power series in $w$ with
coefficients in $\qs[q,\nu,y]$, the ring of polynomials in $q, \nu$ and $y$ with rational coefficients. The expansion of $T(y)$ at order~$2$ reads
\[
 T(y)=w +y \left(q-1+\nu\right) \left(\nu +y \right) w^{2}
+\LandauO(w^3).
\]
The \gf\ of near-triangulations of outer degree $1$ is
\begin{multline*}
T_1:=[y^1] T(y)= \nu(q-1+\nu)w^2+\nu\left((q-1)(q-2+2\nu)
+\nu^2(q-1+\nu^2)\right.\\
\left. +2\nu(q-1+\nu)(q-1+\nu^2)
+\nu^2(q-1+\nu)^2
\right)w^3+O(w^4),
\end{multline*}
as illustrated in Figure~\ref{fig:small-triangulations}.
More generally, we denote by $T_i$ the coefficient of $y^i$ in $T(y)$, that is, the Potts \gf\ of near-triangulations of outer degree $i$. In combinatorial terms,~$T(y)$ counts $q$-coloured near-triangulations by vertices ($w$), 
monochromatic edges ($\nu$) and outer degree~($y$), with the convention that the root vertex is
coloured in a prescribed colour (this accounts for the division by $q$).

\begin{figure}[h]
  \centering
  \includegraphics{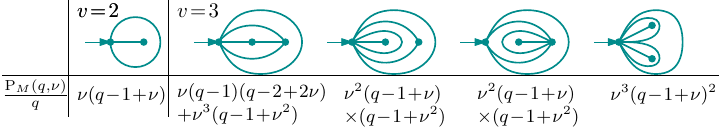}
\caption{The rooted near-triangulations of outer degree 1 with $v=2$
  and $v=3$ vertices, and their Potts polynomials (divided by $q$).
} 
\label{fig:small-triangulations}
\end{figure}

\medskip
\paragraph{\bf Duality.} It follows from the connection between the Potts polynomial and the Tutte polynomial~\cite[Sec.~4.4]{welsh-book} and from the duality property of the Tutte polynomial~\cite[Sec.~3.3]{welsh-book}, that, for a planar map $M$ and its dual $M^*$,
\[
 (\nu_*-1)^{\ff(M)-1} {\Ppol_{M}(q, \nu)}=(\nu-1)^{\ff(M^*)-1} {\Ppol_{M^*}(q, \nu_*)}
\]
where $q=(\nu-1)(\nu_*-1)$. This can be checked for instance on the maps of Figure~\ref{fig:example-map}, for which $\Ppol_M(q,\nu)$ is given by~\eqref{Pott-exemple}, while
\[
  \Ppol_{M^*}(q, \nu_*)=q(q-1+\nu_*) \left((q-1)(q-2)+ (q-1)\nu_*^2+2(q-1)\nu_*+\nu_*^4\right).
\]
If we then define  the  \emm Potts \gf,\ of near-cubic maps by:
\[ 
K(y)\equiv K(q,\nu,w;y)
:=\frac 1 q \sum_{M} \Ppol_M(q,\nu) w^{\ff(M)} y^{\dv(M)},
\] 
where the sum runs over all planar near-cubic maps and $\dv(\cdot)$ denotes the degree of the root vertex, the above duality relation, combined with~\eqref{fe-v}, yields
\[
  K(q,\nu,w;y)=\frac q {(\nu-1)^3} \ T\!\left(q, \nu_*, \frac 1 q (\nu-1)^3 w; \frac y {\nu-1}\right),
\]
where as above $q=(\nu-1)(\nu_*-1)$.
That is, if $K_i(q,\nu, w)$ is the Potts \gf\ of near-cubic maps of root degree $i$, then
\beq\label{eq:KT}
    K_i(q,\nu,w)  
  = \frac q{(\nu-1)^{i+3}} \, T_i\!\left(q,\nu_*, \frac 1 q (\nu-1)^3 w\right).
\eeq
In particular, for $q=3$, $i=1$ and $\nu=0$, we obtain that the series $K_1$ of Corollary~\ref{cor:proper_cubic} is
\[
  K_1(3,0,w)=  3\, T_1\!\left(3,-2, -\frac 1 3 w\right).
\]
 Corollary~\ref{cor:proper_cubic} then follows from~\eqref{alg:dT1}.

\subsection{Power series}
Let $A$ be a commutative ring and $x$ an indeterminate. We denote by
$A[x]$ (resp. $A[[x]]$) the ring of polynomials (resp. \fps) in $x$
with coefficients in $A$. If $A$ is a field, then~$A(x)$ denotes the field
of rational functions in $x$.
These notations are generalized to polynomials, fractions
and series in several indeterminates.
The coefficient of $x^n$ in a power  series $F(x)$ is denoted
by $[x^n]F(x)$.

Recall that a power series $F(x_1, \ldots, x_k) \in \GK[[x_1, \ldots, x_k]]$, where $\GK$ is a
field, is \emm algebraic , (over $\GK(x_1, \ldots, x_k)$) if it satisfies a
non-trivial polynomial equation $P(x_1, \ldots, x_k, F(x_1, \ldots,
x_k))=0$.

 \section{Functional equations}
 \label{sec:eqfunc}

\subsection{Two catalytic variables}
The first way to compute the coefficients of the series $T(y)\equiv T(q,\nu,w;y)$ defined in~\eqref{T-def} (seen as a series in $w$) relies on a functional equation satisfied by a  series $\gQ(x,y)\equiv\gQ(q,\nu, w,t;x,y)$ that counts certain Potts-weighted maps, called \emm quasi-triangulations,,  that are more general than near-triangulations and will not be defined here. This series involves an additional variable $t$ counting edges, and an additional ``catalytic'' variable $x$. The equation, established in~\cite[Prop.~2]{BeBM-11},  reads:
\begin{multline}
  \label{eq:Q}
Q(x,y) = 1 
+ t\, \frac{Q(x,y)-1-yQ_1(x)}{y}+  xt (Q(x,y)-1) + xyt Q_1(x) Q(x,y)
\\
+yt(\nu-1)\gQ(x,y)(2x\gQ_1(x)+\gQ_2(x))
+y^2wt\left(q+ \frac{\nu-1}{1-xt\nu}\right) \gQ(0,y)\gQ(x,y)
\\+\frac{ywt(\nu-1)}{1-xt\nu} \frac{\gQ(x,y)-\gQ(0,y)}{x}
  \end{multline}
where $\gQ_1(x)=[y]\gQ(x,y)$ and $\displaystyle
\gQ_2(x)=[y^2]\gQ(x,y)={(1-2xt\nu)}\gQ_1(x)/({t\nu})$.
This equation defines    a unique power series $Q$ in $t$,
which has polynomial coefficients in
$q, \nu, w, x$ and~$y$. It is said to be \emm catalytic, in $x$ and $y$, because one cannot derive  immediately from it an equation for simpler series in which we could be interested, like
$Q(0,y)$ or $Q_1(x)$.  Divided differences like $(Q(x,y)-Q(0,y))/x$ are sometimes called  \emm discrete derivatives,, which makes the above equation a \emm discrete (partial) differential equation,.

The Potts \gf\  $T(y)$ of near-triangulations is then related to $Q(x,y)$ by:
\[
  T(q,\nu,w;y)\equiv T(y) = w\,Q(q,\nu,w,1;0,y).
\]
According to~\eqref{fe-v}, for $n\ge 2$ one needs to know $\gQ$ up to the coefficient of $t^{3n-4}$ to determine the coefficient of $w^n$ in $T(y)$. In our \Maple\ session we use~\eqref{eq:Q} to compute these coefficients effectively, by induction on $n$. In~\cite{BeBM-17}, the series $T(y)$ is written, alternatively, as $w Q(q, \nu, 1, w^{1/3} ; 0, w^{1/3} y)$. 

\subsection{One catalytic variable}
It was proved in~\cite{BeBM-11} that when $q=3$ (and more generally when $q\not=0,4$ is of the form $q=4\cos^2(k\pi/m)$, for integers $k$ and $m$), the series $T(y)$ is also characterized by an explicit equation involving only \emm one, catalytic variable, namely $y$. 
 Here we write it for  $q=3$, using Proposition 7 in~\cite{BeBM-17} (which is based on~\cite[Cor.~12]{BeBM-11}). We introduce the following notation:
\begin{itemize}
\item $I(y)\equiv I(\nu,w;y)$ is a variant of the Potts \gf\  $T(3,\nu,w;y)$:
\beq\label{I-def}
I(y)=3y\Q(3,\nu,w;y)-\frac{1}{y}+\frac{1}{y^2},
\eeq
\item $N(y,x)$ and $D(x)$ are the following (Laurent) polynomials, where we write $\beta:=\nu-1$:
  \begin{align}
    \bN(y,x)&= \beta /y+3 \nu x +\beta,   \nonumber \\
    \bD(x)&=3\nu^2x^2+\beta(4\nu-1)x- 3\beta\nu w+\beta^2.
            \label{D-expr}
  \end{align}
\end{itemize}
 We moreover denote by $\Tch\!\!_6$ the $6^{\text{\small th}}$ Chebyshev polynomial of the first kind:
\[
\Tch\!\!_6(u)=32 u^6-48u^4+18 u^2-1.
\]

The following proposition is then the case $q=3$ of~\cite[Prop.~7]{BeBM-17}.

\begin{Proposition} \label{prop:eqinv}
  There exist $7$ formal power series in $w$ with coefficients in
$\qs(q,\nu)$, 
denoted $C_0, \ldots$, $C_6$, such that
\beq\label{eq:invT}
\bD(I(y))^{3} \,\Tch\!\!_6\!\left(\frac{\bN(y,I(y))}{2\sqrt
    {\bD(I(y))}}\right)= \sum_{r=0}^6 C_r I(y)^r. 
\eeq
Each series $C_r$ has a rational expression in terms of $\nu, w, T_1, T_3, T_5$ and $T_7$, where $T_i:=[y^i]T(y)$ is the $3$-Potts \gf\ of  near-triangulations with outer degree $i$.
\end{Proposition}

Both sides of~\eqref{eq:invT} expand as polynomials in $T(y)$ and Laurent polynomials in $y$. 
We recall from  Section~13.2 in~\cite{BeBM-11} (or Lemma~8 in~\cite{BeBM-17}) that the expressions of the series $C_r$ in terms of the~$T_i$'s are obtained by 
expanding~\eqref{eq:invT} around $y=0$, up to the coefficient of $y^0$.
 This expansion 
also yields $T_0=w$ and expressions of $T_2, T_4$ and $T_6$ in terms of $T_1, T_3, T_5$. The series $C_4$, $C_5$ and~$C_6$ in fact do not involve any $T_i$:
\begin{align}
   C_6&= -27\nu^6, \nonumber\\
  C_5&=27\nu^4(\nu - 1)(2\nu - 5), \label{Cr-explicit}\\
  C_4&= \frac 9 2 \nu^2(\nu - 1)(18\nu^3w + 35\nu^3 - 75\nu^2 + 30\nu + 10). \nonumber
 \end{align}
The series $C_3$ involves the series $T_1=[y^1]T(y)$:
\beq\label{C3-T1}
C_3=-486 \nu^{4} \left(\nu -1\right)^{2} T_1
+135 \nu^{3} \left(\nu -1\right)^{2} \left(2+\nu \right) w +\left(\nu -1\right)^{4} \left(136 \nu^{2}+43 \nu +1\right).
\eeq
Furthermore, $C_2$ involves $T_1$ and $T_3$, and so on until $C_0$ which involves $T_1, T_3, T_5, T_7$. For the series $T_i$ with even index $i$, one finds:
\beq\label{T24expr}
  \nu T_2=T_1, \qquad \nu^2 T_4= -\left(6 \nu^{2} w +1\right) T_1 +\nu  \left(\nu +1\right) T_3 +\nu  \,w^{2} \left(2+\nu \right),
\eeq
\begin{multline}\label{T6expr}
  \nu^3 T_6=6 \nu^{2} \left(\nu -1\right) T_1^{2}+\left(4 \nu^{3} w +16 \nu^{2} w +7 \nu  w +\nu +2\right) T_1 -\nu  \left(9 \nu^{2} w +\nu^{2}+2 \nu +2\right) T_3\\
  +\nu^{2} \left(2 \nu +1\right) T_5 -\nu  \,w^{2} \left(9 \nu^{2} w +\nu^{2}+4 \nu +4\right) .
\end{multline}
In the end, Equation~\eqref{eq:invT}, once multiplied by $ 2\nu y^7/(27(\nu-1)^{2})$,  rewrites as  a polynomial equation of degree $5$ in $T(y)$, with coefficients in $\qs[\nu,w,T_1, T_3, T_5, T_7,y]$, with the following terms of higher degree:
\begin{multline}\label{eqinv}
0=  2916 \nu^{5} y^{12} \boldsymbol{T(y)}^{5}
  +27  \nu^{3} y^{9} \Big(
\beta \left(37 \nu +17\right) y^{2}-36 \nu  \left(3 \nu +1\right) y +144 \nu^{2}\Big)
  \boldsymbol{T(y)}^{4}\\
  -2\nu y^6
   \Big(486 \boldsymbol{T_1} \,\nu^{4} y^{4}-81 \nu^{3} \left(5 \nu +1\right) w \,y^{4}+486 \nu^{4} w \,y^{3}-\left(56 \nu^{2}+59 \nu +2\right) \beta^{2} y^{4}   \\
   +9 \nu  \beta \left(38 \nu^{2}+40 \nu +3\right) y^{3}-9 \nu^{2} \left(116 \nu^{2}+11 \nu -19\right) y^{2}+486 \nu^{3} \left(3 \nu +1\right) y -972 \nu^{4}
  \Big)
  \boldsymbol{T(y)}^{3}  + \cdots,
\end{multline}
where we have written $\beta=\nu-1$. The complete equation is given in Appendix~\ref{app:1cat}, see~\eqref{1cat-complete}. We also refer to our \Maple\ session where this equation is derived. We use it in Section~\ref{sec:Ti} to prove that all series $T_i$ belong to the extension of $\qs(\nu,w)$ generated by $T_1$.

\subsection{A polynomial system}
\label{sec:polsys}
Equation~\eqref{eq:invT}, or equivalently~\eqref{eqinv}, is an equation in a single catalytic variable, $y$. In~\cite[Sec.~11]{BeBM-11}, it was derived from this equation, using the general results of~\cite{mbm-jehanne}, that  $T(y)$ is algebraic over $\qs(\nu,w,y)$. Moreover,~\cite{mbm-jehanne} also shows that
an annihilating polynomial of $T(y)$ can be produced by computing a Gr\"obner basis for some (big) ideal. However, this approach fails here because of the large
size of the polynomials generating this  ideal.  Beyond the original  approach of~\cite{mbm-jehanne}, more and more efficient techniques  have been  designed to solve such catalytic equations~\cite{BCNS_ISSAC22,BNS_ISSAC23,Notar-DDE}. But, in addition to difficulties due to the sizes of all systems, it appears that a non-degeneracy condition needed to apply these techniques does not hold for this problem~\cite[Sec.~12.6]{Notarantonio-thesis}. As a result, Equation~\eqref{eqinv}  has so far resisted all attempts, and the minimal polynomial of $T(y)$ (or even~$T_1$) has remained  out of reach.

In this paper, we determine the minimal polynomial of $T_1$ (and in fact of each series $T_i$ for $i\le 7$)
by starting from a polynomial system that is smaller and  better structured than those derived from the above  general methods. This alternative system was established in~\cite{BeBM-17}, 
in the process of deriving a system of \emm differential, equations defining $T_1$ (and valid for any $q$).
We give at the end of the section a few details on the connection between this system and the common basis to the general methods.

\begin{Proposition}\label{prop:factor}
  Let $C(x)= C_0 + \cdots + C_6 x^6$, where the series $C_r$ are those of Proposition~\ref{prop:eqinv},
  and let $D(x)$ be defined by~\eqref{D-expr}. There exist $4$ formal power series $X_1, \ldots, X_4$ in $w$, with coefficients in an algebraic closure of $\qs(\nu)$ and constant terms distinct from $0, -1/4$ and $-1/(2\nu)$, and a pair
  $\big(\hP_-(x), \hP_+(x)\big)$ of
  polynomials in $x$ with coefficients in $\rs(\nu,w)[[t]]$
  such that
 \begin{align}
 \bC(x)-\bD(x)^3&= \displaystyle \hP_-(x)
 \prod_{i=1}^{2}(x-X_i)^2,
\label{factor1+}
\\
 \bC(x)+\bD(x)^3&= \displaystyle \hP_+(x)
 \prod_{i=3}^{4}(x-X_i)^2.
\label{factor1-}
 \end{align}
\end{Proposition}

\begin{Remark}
  The above equations form a polynomial system of $8$ equations relating $C_0, \ldots, C_6, \allowbreak X_1, \ldots, X_4$, once written as
  \begin{align*}
    C(X_i)=D(X_i)^3, &\quad C'(X_i)=3 D'(X_i) D(X_i)^2, \qquad i=1, 2,
    \\
    C(X_i)=-D(X_i)^3, &\quad C'(X_i)=-3 D'(X_i) D(X_i)^2, \quad                  \      i=3,4,
  \end{align*}
 where the derivatives are taken with respect to $x$. Since $C_4, C_5$ and $C_6$ are explicit, see~\eqref{Cr-explicit}, there are exactly $8$ unknown series. 
 The above system implies the existence of a polynomial~$\hQ(x)$ such that
 \beq
 \bD(x)\bC'(x)- 3 \bD'(x)\bC(x)= \displaystyle \hQ(x) \prod_{i=1}^{4}(x-X_i).\label{factor2}
 \eeq
 This equation occurs as well in~\cite[Prop.~11]{BeBM-17}.
\end{Remark}

\begin{proof}
  The case $m=6$ of Proposition~11 in~\cite{BeBM-17} shows the existence of $4$ series $X_i$ that satisfy~\eqref{factor2} and also
  \[
    \bC(x)^2-\bD(x)^6=\left(\bC(x)-\bD(x)^3\right)\left(\bC(x)+\bD(x)^3\right)  = \hP(x) \prod_{i=1}^{4}(x-X_i)^2,
  \]
  for some polynomial $\hP(x)$.
  The conditions on the constant terms of the $X_i$ arise from~\cite[Lem.~9]{BeBM-17}. It remains to refine the above equation into~\eqref{factor1+} and~\eqref{factor1-}. One way to do this is to dig into the details of the proof of~\cite[Prop.~11]{BeBM-17}. Another way is to examine the first coefficients of the roots of the left-hand side of~\eqref{factor2}, and decide whether they solve~\eqref{factor1+} or~\eqref{factor1-}. Indeed,
   the series $C_r$ are explicit in terms of the coefficients~$T_i$ of $T(y)$, and we can compute inductively the coefficient of $w^n$ in $T(y)$,
  so we can also determine the first terms of the roots 
  of $\bD(x)\bC'(x)- 3 \bD'(x)\bC(x)$. This polynomial has degree~$6$ in $x$, hence $6$ roots, which we find to start as follows (of course, the labelling is chosen so as to satisfy~\eqref{factor1+} and~\eqref{factor1-} in a near future):
  \allowdisplaybreaks
\begin{align*}
  X_1&=\frac{1-\nu}{\nu^{2}}+\nu  \left(2 \nu -1\right) w + \frac{2\nu^{3} \left(6 \nu^{3}-6 \nu^{2}-3 \nu +4\right)}{\nu -1} w^{2}+\LandauO(w^{3}),
  \\
  X_2&=\frac{1-\nu}{2 \nu}-w - \frac{2\nu  \left(\nu^{2}+3 \nu -3\right)}{\nu -1} w^{2}+\LandauO(w^{3}),
  \\
  X_3&=\frac{\left(1-\nu \right) \left(8 \nu +1+\sqrt{1+16\nu-8 \nu^{2}}\right)}{18 \nu^{2}}+\frac 3 2 \left(\frac{2 \nu +1}{ \sqrt{1+16\nu-8 \nu^{2}}}-1\right) w +\LandauO(w^{2}),
  \\
  X_4&=\frac{\left(1-\nu \right) \left(8 \nu +1-\sqrt{1+16\nu-8 \nu^{2}}\right)}{18 \nu^{2}}-\frac 3 2 \left(\frac{2 \nu +1}{ \sqrt{1+16\nu-8 \nu^{2}}}+1\right) w +\LandauO(w^{2}).
\end{align*}
The last two roots
have constant terms $0$ and $-1/(2\nu)$, so they cannot be  any of the $X_i$'s.

Now it suffices to plug the above series $X_i$ in the polynomials $\bC(x)-\bD(x)^3$ and $\bC(x)+\bD(x)^3$ to see that $X_1$ and $X_2$ cannot be roots of $\bC(x)+\bD(x)^3$, while $X_3$ and $X_4$ cannot be roots of  $\bC(x)-\bD(x)^3$.
This yields the final forms~\eqref{factor1+} and~\eqref{factor1-}.
\end{proof}

From now on, we will be only interested in the elementary symmetric functions of $X_1$ and $X_2$ on the one hand, and of $X_3$ and $X_4$ on the other hand. They appear to have coefficients in~$\qs(\nu)$, and this will be proved in the next subsection:
\begin{align*}
  S_1&:=X_1+X_2 \nonumber 
       \\ &=-\frac{\left(\nu +2\right) \left(\nu -1\right)}{2 \nu^{2}}+\left(2 \nu +1\right) \left(\nu -1\right) w +6 \left(\nu -1\right) \left(\nu +1\right) \left(2 \nu^{2}+1\right) \nu  w^{2}+\LandauO(w^{3})  ,\nonumber 
  \\
  P_1&:=X_1X_2 \nonumber 
  \\ &=\frac{\left(\nu -1\right)^{2}}{2 \nu^{3}}
       - \frac{\left(\nu -1\right) \left(2 \nu^{3}-\nu^{2}-2\right)}{2\nu^{2}} w
       - \frac{3\left(2 \nu^{4}+2 \nu^{3}+\nu^{2}+2 \nu +2\right) \left(\nu -1\right)^{2}}{\nu} w^{2}+\LandauO(w^{3})   ,\nonumber 
  \\
  S_3&:=X_3+X_4\label{Pi-Si}
       \\ &= -\frac{\left(8 \nu +1\right) \left(\nu -1\right)}{9 \nu^{2}}-3 w -6\left( \nu^{2}+4 \nu +1\right) w^{2}  +\LandauO(w^{3}) ,\nonumber 
  \\
  P_3&:=X_3X_4 \nonumber  \\ 
     & =\frac{2 \left(\nu -1\right)^{2}}{9 \nu^{2}}
       + \frac{\left(5 \nu +1\right) \left(\nu -1\right)}{3\nu^{2}} w
       + \frac{\left(\nu -1\right) \left(7 \nu^{3}+30 \nu^{2}+24 \nu +2\right)}{3\nu^{2}} w^{2} +\LandauO(w^{3}) .  \nonumber 
\end{align*}

\subsection{Connection with the standard approach of 1-catalytic equations}
When studying a 1-catalytic equation like~\eqref{eqinv}, written as $\Pol(T(y),y,T_1,T_3,T_5,T_7,w)=0$, the key idea is to examine the series $Y\equiv Y(w)$ such that $\Pol'_1(T(Y),Y,T_1,T_3,T_5,T_7,w)=0$, where $\Pol'_1$ denotes the derivative of $\Pol$ with respect to its first variable~\cite{mbm-jehanne}. By the chain rule, this also implies $\Pol'_2(T(Y),Y,T_1,T_3,T_5,T_7,w)=0$. Here, one finds, using an expansion near $w=0$, that seven such series exist, say $Y_0, Y_1, \ldots, Y_6$. All effective strategies in the literature  exploit in one way or another the $3\times 7$ equations $\Pol=\Pol'_1=\Pol'_2=0$, when evaluated at $(T(Y_i),Y_i,T_1,T_3,T_5,T_7,w)$.

One of the series $Y$ starts
\[
  Y_0=\nu+2\nu^4w+ \LandauO(w^2).
\]
The other six series have constant terms that are quadratic in $\nu$, and thus go by pairs:
\[
  Y_{1,2}=\frac{4 \nu -1\pm\sqrt{-8 \nu^{2}+16 \nu +1}}{4 \nu -4}+ \LandauO(w),
\]
\[
  Y_{3,4}=\frac{2 \nu +1\pm\sqrt{-8 \nu^{2}+16 \nu +1}}{2 \nu -2}+ \LandauO(w),
\]
\[
  Y_{5,6}=\frac{\nu \pm\sqrt{\nu  \left(2-\nu \right)}}{\nu -1}
+ \LandauO(w).
\]

Let us explain how these seven series relate to the four series of~\cite[Lem.~9]{BeBM-17}. First, recall that the equation  $\Pol(T(y),y,T_1,T_3,T_5,T_7,w)=0$ is essentially obtained by multiplying~\eqref{eq:invT} by $y^7$, where $I(y)$ is the variant of $T(y)$ defined by~\eqref{I-def}.
  From this, and the chain rule, it follows that \emm at least  some of the, $Y_i$  are obtained the roots of
  \beq \label{sin}
    N(y, I(y))^2/ D(I(y)) \in \left\{ 4 \cos^2 (j\pi/6), j=1, \ldots, 3\right\}=
    \{0, 1, 3\}.
  \eeq
  More precisely, any such root $Y$ is for sure one of the $Y_i$, provided that it does not cancel $\partial_y I(y)$ (since $\partial_y I(y)$ is $(y-2)/y^3$ at $w=0$, the latter condition holds automatically as soon as the constant coefficient of  $Y$ is neither $0$ not $2$). 
  Lemma~9 in~\cite{BeBM-17} precisely studies the above equation~\eqref{sin} (albeit for more general values of $q$), and establishes in our case the existence of four roots $Y$ whose constant term, expanded in $\beta:=\nu-1$, is $1+\LandauO(\beta)$.  These four roots thus coincide with the series $Y_0$,
  $Y_2$,
  $Y_4$
  and $Y_6$
  above. Moreover, it follows from the proof of~\cite[Lem.~9, p.~62]{BeBM-17} that in each of these series, all coefficients belong to the field generated over $\qs(\nu)$ by the \emm constant, coefficient. Hence $Y_0$ has rational coefficients, $Y_6$
  has coefficients in $\qs(\nu, \sqrt{\nu  \left(2-\nu \right)})$, while $Y_2$ and $Y_4$  have coefficients in $\qs(\nu, \sqrt{-8 \nu^{2}+16 \nu +1})$.
We claim that 
  the remaining three series, namely $Y_1$, $Y_3$ and $Y_5$, are just obtained by taking the algebraic conjugates (i.e., changing the sign before the square roots) of the coefficients in $Y_2$, $Y_4$ and $Y_6$. Indeed, given that $N(y,x)^2/D(x)$ has rational coefficients, these three series also solve~\eqref{sin}, and thus provide three more solutions of  $\Pol'_1(T(Y),Y,T_1,T_3,T_5,T_7,w)=0$. We  thus recover all series $Y_i$ as solutions of~\eqref{sin}.

It follows from~\cite[Prop.~11]{BeBM-17} that the four series $X_i$ of Proposition~\ref{prop:factor} are the values $I(Y_i)$, for $i \in \{0, 2, 4, 6\}$, where $I(y)$ is the variant of $T(y)$ defined by~\eqref{I-def}. More precisely, we find by expanding series near $w=0$ that
\[
  I(Y_0)=X_1, \quad  I(Y_2)  =X_4, \quad 
  I(Y_4)=X_3, \quad 
  I(Y_6)=X_2.
\]
But the proof of~\cite[Lem.~10]{BeBM-17}, repeated for the remaining three series $Y_i$, which are also solutions of~\eqref{sin}, implies that the corresponding series  $I(Y_i)$ are also double roots of $C(x)^2-D(x)^6$. It is easy to see, by inspecting its first coefficients, that this polynomial does not have more than  four double roots (namely, the $X_j$), and consequently, all series $I(Y_i)$ are in $\{X_1, X_2, X_3, X_4\}$. Using the first few coefficients of the $X_j$ and $Y_i$, we  complete the above identities as
\[
  I(Y_0)=X_1,
   \quad I(Y_2) =I(Y_3)  =X_4,
  \quad  I(Y_1)=  I(Y_4)=X_3,
   \quad I(Y_5)=  I(Y_6)=X_2.
\]
The polynomial system of Proposition~\ref{prop:factor} gives a set of compact relations between the $X_i$, not involving any of the $Y_i$. It seems that handling the four series $X_i$ rather than the seven series~$Y_i$ avoids some redundancy.

The above  connection with~\cite{BeBM-17} can be used to prove that the series $S_1$, $P_1$, $S_3$ and $P_3$ defined in terms of the $X_i$ at the end of Section~\ref{sec:polsys} have coefficients in~$\qs(\nu)$, as claimed above. Let us recall that $Y_0$ has rational coefficients, while all other $Y_i$ have coefficients in a quadratic extension of $\qs(\nu)$. Given that $I(y)$ has rational coefficients, these properties are preserved in $I(Y_i)$. In particular, $X_1=I(Y_0)$ has rational coefficients. Let us now write $2 X_2= I(Y_5)+I(Y_6)$. Given that the coefficients of $Y_5$ and $Y_6$ are algebraic conjugates, we conclude that $X_2$ also has rational coefficients. A similar conjugacy argument proves that
  \[
    S_3= X_3+X_4 = \frac 1 2 \left(I(Y_1)+I(Y_2)+I(Y_3)+I(Y_4) \right)
  \]
  and
   \[
    P_3= X_3X_4 = \frac 1 4 \left(I(Y_1)+I(Y_4)\right)\left(I(Y_2)+I(Y_3) \right)
  \]
  have rational coefficients in $\nu$.

\section{Derivation of the series $\boldsymbol{T(y)}$}
\label{sec:Ty}

  In this section we determine the series $T(y)$, starting from the system of Proposition~\ref{prop:factor}.

  \subsection{Minimal polynomial of the series $\boldsymbol{S_1=X_1+X_2}$ (direct version)}
  \label{sec:elim-direct}

\subsubsection{A system of four equations for $\boldsymbol{S_1, P_1, S_3}$ and $\boldsymbol{P_3}$}
\label{subsec:rem}
The first equation of Proposition~\ref{prop:factor} can be written
\beq\label{rem:minus}
  \rem\big(\bC(x)-\bD(x)^3, (x^2-S_1 x + P_1)^2 , x\big)=0,
\eeq
where, given two polynomials $A$ and $B$ in $x$,  $\rem(A,B,x)$ is the remainder of {the Euclidean division of $A$ by $B$}, and $S_1, P_1$ are the sum and product of $X_1$ and $X_2$. The above remainder has degree~$3$ in $x$, so its coefficients give four  polynomial equations relating the series $S_1$, $P_1$, $C_0$, $C_1$, $C_2$, and~$C_3$ (recall from~\eqref{Cr-explicit} that the other series $C_r$ are explicit).
We proceed similarly with the second
equation of Proposition~\ref{prop:factor}.
This gives a system of $2\times4=8$ polynomials relating the $8$ series $S_1$, $S_3$, $P_1$, $P_3$ and  $C_r$, for $0\le r\le 3$.
Recall  from~\eqref{C3-T1} that we are mostly interested in $C_3$, since it is closely related to $T_1$.

{However, our first step will be to eliminate the four series $C_r$.}  This may seem counter-intuitive,
but turns out to work well. It is in fact very natural, as the system obtained from~\eqref{rem:minus} is linear in the $C_r$. We solve it in these unknowns and obtain an expression of each $C_r$ as a polynomial in $S_1$, $P_1$, $w$ and $\nu$. Now we replace each $C_r$ by its expression in the second group of four equations, derived from the second equation of Proposition~\ref{prop:factor}. We thus obtain a system of four equations where the only unknowns are $S_1, P_1, S_3$ and $P_3$. 

\subsubsection{Elimination}
\label{subsec:elim1}
We now perform a  step-by-step elimination procedure based on resultants, in which we exploit the fact that we know the first coefficients of all series involved in the system. In this way, when we find a polynomial relation between our four series that factors,
we remove from it the factors
 that, given the first few terms of the series, cannot vanish.

 Starting  from the system that we have just obtained, we eliminate first $P_3$, then $P_1$, then~$S_3$, to obtain the minimal polynomial of $S_1$. Let us give a few details. The smallest of the four equations we start from contains $36$ monomials (in $\nu, w$ and the four unknown series),  has degree~$1$ in each~$P_i$,  degree $3$ in  $S_1$ and $2$ in $S_3$.
 We take its resultant, in turn, with each of the other three polynomials, with respect to $P_3$, and factor it.  Only one of the factors that occur depends on the unknown series, so this is  the factor that vanishes (this will be different with the lighter approach discussed in Section~\ref{sec:light}, where non-trivial factorizations occur at each stage).

 Now we repeat the procedure by eliminating $P_1$ between the smallest of these three equations ($84$ terms) and each of the others. Again, only trivial factorizations occur.
 A final elimination of~$S_3$ between the two remaining equations
 gives an annihilating polynomial for $S_1$. This time, it factors into five terms involving $S_1$, and moreover one of them is cubed.  We decide from the first coefficients of~$S_1$ which of these factors is the minimal polynomial of $S_1$.

We find  $S_1$  to be algebraic of degree $11$ over $\qs(\nu,w)$. Its minimal polynomial contains $394$ monomials, when seen as a polynomial in $\nu, w$ and $S_1$, but only $36$ as a polynomial in $w$ and $S_1$. It has degree $21$ in $\nu$, $4$ in~$w$, and of course $11$ in $S_1$.

\subsection{Minimal polynomial of the series $\boldsymbol{S_1=X_1+X_2}$ (lighter version)}
\label{sec:light}
We discuss here an alternative derivation of the minimal polynomial of $S_1$, which uses, in addition to the two equations of Proposition~\ref{prop:factor}, the extra equation~\eqref{factor2}.  In this way, we  obtain more equations than needed, but in the elimination procedure, the resultants systematically factor and this results into faster and lighter computations.

\subsubsection{A system of eight equations for $\boldsymbol{S_1, P_1, S_3}$ and $\boldsymbol{P_3}$}
\label{sec:elim-Ci}
Let us rewrite~\eqref{factor2} as
\[
  \rem\big(\bD(x)\bC'(x)- 3 \bD'(x)\bC(x),
  (x^2-S_1 x + P_1) (x^2-S_3 x + P_3) , x\big)=0.
\]
This gives a system of four linear equations in the four series $C_0, \ldots, C_3$. We check that its determinant is non-zero, using the first coefficients of the $S_i$ and $P_i$.  Solving this linear system gives expressions of $C_0, \ldots, C_3$ as rational functions in $\nu$, $w$, $S_1$, $S_3$, $P_1$ and $P_3$.

We now replace each $C_i$ by its expression in the $8$ polynomial equations obtained from Proposition~\ref{prop:factor} at the beginning of Section~\ref{subsec:rem}. We thus obtain a system of eight equations where the only unknowns are $S_1, P_1, S_3$ and $P_3$. Clearly we have ``too many'' equations, but this  gives us some leeway to choose the smallest ones,  when convenient.

\subsubsection{Elimination}

As in Section~\ref{subsec:elim1},  we will eliminate first $P_3$, then $P_1$, then~$S_3$, to obtain the minimal polynomial of $S_1$. Let us give a few details. The smallest of the eight equations contains $366$ monomials (in $\nu, w$ and the four unknown series),  has degree~$2$ in each $P_i$, and degree $3$ in each $S_i$.
We take its resultant, in turn, with each of the other seven polynomials, with respect to $P_3$. Each of these seven resultants is found to factor, and in each case, we prove using the first coefficients of $S_1, P_1$ and $S_3$ that only one factor vanishes: of course this is the factor that we retain to proceed with further eliminations. Three of these resultants yield the same factor, so at the end of this step, we have five equations between $S_1, P_1$ and $S_3$.

Now we repeat the procedure by eliminating $P_1$ between the smallest of these five equations ($170$ terms) and each of the others. In each resultant, we only retain the (unique) vanishing factor. This gives a system of four equations between $S_1$ and $S_3$. A final elimination of $S_3$ between two of these equations gives an annihilating polynomial for $S_1$. We thus recover
the minimal polynomial of $S_1$. 

Compared to the approach of Section~\ref{sec:elim-direct}, the main advantage is that non-trivial factorizations occur at earlier stages (in fact, in all resultants that we compute). In particular, the final elimination of $S_3$ in the first approach involves polynomials of total degree $20$ and $22$ in the $S_i$, involving $1508$ and $2756$ monomials, respectively. With the lighter approach, the final elimination of $S_3$  involves polynomials of total degree $14$ and $21$ in the $S_i$, involving $510$ and $1580$ monomials, respectively.

\subsection{Elliptic parametrization of $\boldsymbol{(w,S_1)}$}
\label{sec:elliptic}

The curve $\cC:=\{(w, S_1)\}$ (with $\nu$ as a parameter) is found to have genus $1$, that is, to be elliptic (we use the {\tt algcurves} package in \Maple). This implies that it can be parametrized by writing~$w$ and $S_1$ as rational functions in some parameter $U$ and the square root of some polynomial in~$U$ (with coefficients in $\overline{\qs(\nu)}$). This is the natural counterpart for elliptic curves of the rational parametrization of curves of genus $0$. One can then hope that
the series $S_3$, $P_1$ and $P_3$ can be expressed in terms of $U$ (and a square root) as well. This will indeed turn out to be the case.

Since the minimal polynomial of $S_1$ has degree $4$ in $w$, and not $2$, the series $S_1$ itself cannot be used as the parametrizing series $U$. However, we were able to determine a suitable parametrization using {\sc Maple}. Later, having computed the minimal polynomial of  $\partial_w T_1$ (for reasons that are explained below~\eqref{alg:dT1}), we discovered that this series could be used as parametrizing series,
and this is what we will do below. But let us briefly explain how we first constructed a parametrizing series $U$, since this can be of interest to some readers. Details are available in the \Maple \ session accompanying this paper.

Using the
command {\tt Weierstrassform}, we first constructed a Weierstrass form of the curve~$\cC$ for various values of $\nu$, and were then able to conjecture from them a generic form, valid for an indeterminate $\nu$. This conjecture stated that the curve $\cC$ was birationally equivalent to the curve
\beq\label{courbe}
U^{3}+\frac{\left(1-\nu \right) V^{2}}{2}+3 \left(7 \nu^{2}-14 \nu -29\right) U +\frac{2 \left(5 \nu^{2}-10 \nu -13\right)^{2}}{\nu -1}=0.
\eeq
Still with {\sc Maple}, one can also obtain, for a fixed value of $\nu$, rational expressions of $w$ and $S_1$ in terms of $U$ and $V$ lying on the above curve. These expressions read, respectively,
\[
  w= \frac{A_9(U) V + A_{10}(U)}{D_{11}(U)}, \qquad S_1= \frac{A_2(U) V + A_{4}(U)}{D_4(U)},
\]
where the $A_i$ and $D_i$ are polynomials whose degrees are indicated by their subscripts. Having determined them for sufficiently many values of $\nu$, we could derive by rational interpolation (conjectural) expressions of~$w$ and $S_1$ in terms of $U$ and $V$, valid for any $\nu$. Finally, to prove these conjectured expressions, we just had to replace $w$ and $S_1$ by these expressions in the minimal polynomial of $S_1$, and check that this was $0$ on the curve~\eqref{courbe}.

But from now on however, we will use as parametrizing series  the only solution of~\eqref{alg:dT1} that has constant term $0$, denoted by $\DT$. Of course at this stage we do not know that this is the $w$-derivative of $T_1$ (but we will prove it below). Solving~\eqref{alg:dT1} for $w$ gives a rational expression of~$w$ in terms of $\DT$ and $\sqrt\Delta$, where
\begin{multline*}
  \Delta=  \nu\left( 144 \nu^{3} \DT^{4}-24 \nu^{2} \left(7 \nu +5\right) \DT^{3}+6 \nu  \left(13 \nu^{2}+16 \nu +7\right) \DT^{2} \right.
  \\ \left.
    -2\left(8 \nu^{3}+15 \nu^{2}+12 \nu +1\right) \DT +\nu  \left(2+\nu \right)^{2}\right).
\end{multline*}
This is a bit bigger than the square root arising from~\eqref{courbe}, but in fact the minimal polynomial of~$U$ has much more terms than~\eqref{alg:dT1}. We now replace $w$ by this expression in the minimal polynomial of $S_1$, and factor the resulting expression over $\cs(\nu, \DT, \sqrt{\Delta})$: we obtain two factors, and the one that actually vanishes has degree $1$ in $S_1$. This gives the following expression:
\begin{multline}
    S_1=-\frac{3 (1-2 \DT ) \sqrt{\Delta}}{2 \nu  (1+5\nu-12 \nu\DT )}
 \\ -  \frac{ 72 \DT^{3} \nu^{3}-12 \nu^{2} \left(5 \nu +4\right) \DT^{2}+2 \nu  \left(4 \nu^{2}+10 \nu +13\right) \DT +\left(2+\nu \right) \left(2 \nu^{2}-4 \nu -1\right)}{2\nu^2(1+5\nu-12 \nu\DT )}.\label{S1-DT}
\end{multline}

\subsection{The series $\boldsymbol{T_1}$}
\label{sec:T1}

We now make our way backwards in the elimination process that led to the minimal polynomial of $S_1$. In the smallest equation that we had between $S_1$ and $S_3$, we replace $w$ and $S_1$ by their expressions in terms of $\DT$ and $\Delta$, factor the result, observe that the factor that vanishes has degree $1$ in $S_3$, and thus obtain a rational expression of $S_3$ in terms of $\DT$ and $\Delta$, similar to~\eqref{S1-DT}. Then we proceed similarly with the series $P_1$ and finally $P_3$. Both are found to lie in $\qs(\nu, \DT, \sqrt{\Delta})$.

The next step is to return to $C_3$, which is closely related to $T_1$ (see~\eqref{C3-T1}) and was expressed as a polynomial in $S_1$ and $P_1$
  in Section~\ref{subsec:rem}. This gives an expression of $C_3$, then~$T_1$, in $\qs(\nu, \DT, \sqrt{\Delta})$.

It remains to derive from this expression  that the $w$-derivative of $T_1$ is indeed the series $\DT$. We proceed as follows:
\begin{itemize}
\item using the expression of $w$ in terms of $\DT$ and $\sqrt{\Delta}$, we express $T_1$ rationally in terms of~$\DT$ and $w$ rather than $\DT$ and $\sqrt{\Delta}$,
\item we differentiate this in $w$ to obtain an expression of $\partial_w T_1$ in terms of $w$, $\DT$, and $\partial_w \DT$,
\item we differentiate the minimal polynomial~\eqref{alg:dT1} of $\DT$ to obtain an expression of $\partial_w \DT$ in terms of $w$ and $\DT$,
\item combining the last two steps, we obtain an expression of  $\partial_w T_1$ in terms of $w$ and $\DT$,
  \item we reduce it modulo the minimal polynomial of $\DT$ and conclude that  $\partial_w T_1=\DT$.
  \end{itemize}

  Using the first point above, and the minimal polynomial of $\DT$, we also compute the minimal polynomial of $T_1$, which will be useful later. It has degrees $13$ and $27$ in $w$ and $\nu$, respectively.

The above proof that $\partial_w T_1=\DT$ is immediate in \Maple. A more natural, but also significantly longer route,  goes as follows:
\begin{itemize}
  \item we differentiate the minimal polynomial of $T_1$ in $w$ to obtain a rational expression of $\partial_w T_1$ in terms of $w$ and $T_1$,
  \item we replace $\DT$ by this expression in the minimal polynomial~\eqref{alg:dT1} of $\DT$, and reduce the numerator of the result modulo the minimal polynomial of $T_1$. This  gives  $0$. Since $\DT$ is the unique solution of~\eqref{alg:dT1} with constant term $0$, and $\partial_wT_1$ has constant term $0$, we  thus conclude that  $\partial_w T_1=\DT$.
  \end{itemize}
  Note however that we use in both steps the minimal polynomial of $T_1$, which is much bigger than that of $\DT$. The first method performs similar calculations, in steps 3 and 5, with the latter smaller polynomial.  This makes the second method  much longer than the first one.

  \subsection{The series $\boldsymbol{T_i}$ for $\boldsymbol{i\ge 2}$}
\label{sec:Ti}
Recall from Section~\ref{subsec:rem} that we have expressed the series $C_0, \ldots, C_3$ as polynomials in $w$, $\nu$,
$S_1$, $P_1$, $S_3$ and $P_3$.
In Section~\ref{sec:T1}, we have expressed these four series as elements of $\qs(\nu, \DT, \sqrt{\Delta})$.  In Section~\ref{sec:elliptic}, we had obtained such an expression for $w$ as well. So each $C_i$ can now be written as an element of $\qs(\nu, \DT, \sqrt{\Delta})$.

  We have already exploited the fact that $C_3$ is closely related to $T_1$ to determine $T_1$. We now proceed similarly for $T_3$, $T_5$, and $T_7$, in this order, using the fact that $C_2$ (resp. $C_1$, resp. $C_0$) is a polynomial in $T_1$ and $T_3$ (resp. in $T_1, T_3$ and $T_5$, resp.  in $T_1, T_3, T_5$ and $T_7$), with coefficients in $\qs(\nu,w)$, of degree $1$ in $T_3$ (resp. $T_5$, resp. $T_7$) --- see  the discussion following Proposition~\ref{prop:eqinv}. So now each of the series $T_3$, $T_5$, $T_7$ can be written as an element of $\qs(\nu, \DT, \sqrt{\Delta})$.

  Let us now discuss the series $T_2, T_4$ and $T_6$. We recall from~\eqref{T24expr} and~\eqref{T6expr} that they have polynomial expressions in terms of $T_1, T_3$ and $T_5$. This yields expressions for these series in $\qs(\nu, \DT, \sqrt{\Delta})$ as well. Recall also that $T_0=w$.

  Now let us write
  \[
    T(y)= w +T_1 y +T_2 y^2+ \cdots + T_7 y^7 + y^8 S(y),
  \]
  with $T_2, T_4, T_6$ expressed in terms of $T_1, T_3, T_5$, for a series $S(y)$ in $\qs(\nu, y)[[w]]$. We inject this expression in the 1-catalytic equation~\eqref{eqinv} satisfied by $T(y)$. This makes the  equation tautological up to the order of $y^8$, that is, it contains a factor $y^8$. Removing this factor leaves a polynomial equation (of degree $5$) for $S(y)$, with coefficients in $\qs[\nu,w,T_1, T_3, T_5, T_7,y]$. This equation reads
  \[
    36 \nu^{15} S(y) + \Pol_0(\nu,w,T_1, T_3, T_5, T_7) + y \times  \Pol(\nu,w,T_1, T_3, T_5, T_7,y, S(y))=0.
  \]
  This form implies, by induction on $i\ge 0$, that the coefficient of $y^i$ in $S(y)$, that is, the series~$T_{i+8}$, is an element of  $\qs(\nu)[w,T_1, T_3, T_5, T_7]$, and thus of  $\qs(\nu, \DT, \sqrt{\Delta})=\qs(\nu,w,\DT)$. Given that the degree of $\DT$ over $\qs(\nu,w)$ is prime, each series $T_i$ is either rational in $\nu$ and $w$, or algebraic of degree $11$ (since the degree of a sub-extension divides the total degree~\cite[Ch.~V, \S 1, Prop~1.2]{lang}).  The former possibility  will be ruled out by an asymptotic argument in Section~\ref{sec:asympt-triang}; see the proof of Lemma~\ref{lem:Ti}. This completes the proof  of Theorem~\ref{thm:main}.

  \subsection{The Potts model on near-cubic maps}
  We can now, using the change of variables~\eqref{eq:KT} for $q=3$, state a result analogous to Theorem~\ref{thm:main} for  the  3-Potts model on  near-cubic maps.

\begin{Corollary}\label{cor:cubic}
 For $i\ge 1$,  the series $K_i\equiv K_i(\nu,w)$ that counts $3$-coloured near-cubic maps with root vertex of degree $i$ (by monochromatic edges and faces) is algebraic of degree $11$. All  series~$K_i$ belong to the same extension of degree $11$ of $\qs(\nu, w)$.
\end{Corollary}

In particular, one derives from~\eqref{alg:dT1} the minimal polynomial of $\DK:=\partial_w K_1$. It has degree~$2$ in $w$ again, but degree $15$ in $\nu$. The specialization $\nu=0$ then leads Corollary~\ref{cor:proper_cubic}, as already established at the end of Section~\ref{sec:q-Potts}.

\section{Asymptotic results}
\label{sec:asympt}

\subsection{Triangulations}
\label{sec:asympt-triang}
We now study the dominant singularities of the series $T_i$, seen as power series in $w$ depending on a non-negative parameter $\nu$. The singularity analysis of algebraic series in
$\rs[[w]]$
has become  quasi-automatic~\cite[Chap.~VII.7]{flajolet-sedgewick}, but
things are   more delicate here because of the parameter~$\nu$. This is of course a recurrent difficulty in many counting problems.
We refer for earlier (and somewhat smaller) instances to~\cite{Bernardi-Curien-Miermont,albenque-Ising,Chen-Turunen-phase,Chen-trees}. We will use, and sometimes make more systematic, some of the ideas of  these papers.

\begin{Proposition}\label{prop:asympt}  
  Let $\Delta_1$ and $\Delta_2$ be the  polynomials in $\nu$ and $\rho$ given by~\eqref{Delta1} and~\eqref{Delta2} in Appendix~\ref{sec:app-triang}.
  Figure~\ref{fig:radius} shows, among other curves, a plot of
  the curves $\Delta_1(\nu,\rho)=0$ and $\Delta_2(\nu,\rho)=0$.  

  Let $i\ge 1$. Consider $T_i(\nu,w)\equiv T_i$ as a series in $w$ depending on the
  parameter $\nu>0$. Let~$\rho_\nu$ denote its radius of
  convergence. Then $\rho_\nu$ is a continuous non-increasing function of
  $\nu$ for $\nu> 0$, which satisfies
  \begin{align*}
    \Delta_1(\nu, \rho_\nu)&=0 \quad \hbox{for}\quad 0 < \nu \le
                             \nu_c:=1+3/\sqrt{47},\\
    \Delta_2(\nu, \rho_\nu)&=0 \quad \hbox{for} \quad\nu_c \le \nu.
  \end{align*}
  More precisely, between $0$ and $\nu_c$ the radius $\rho_\nu$ is the branch of $\Delta_1(\nu, \rho)=0$ that starts at $\rho_0:=1/8$ when $\nu=0$, and beyond $\nu_c$ the radius $\rho_\nu$ is the highest of the two branches of  $\Delta_2(\nu, \rho)=0$ that start at
  \beq\label{rhoc}
  \rho_{\nu_c}= \frac{1295\sqrt{47}-7875}{109744}.
  \eeq
  Moreover, $T_i$ has no  singularity other than the radius on its circle of convergence.
  For $\nu\neq\nu_c$, the behaviour of $T_i$ near $w=\rho_\nu$ is the standard  singular behaviour of  planar maps series:
  \beq\label{sing_generic}
  T_i= \alpha_{i,\nu} +\beta_{i,\nu}(1-w/\rho_\nu)+\gamma_{i,\nu} (1-w/\rho_\nu)^{3/2}\,(1+o(1)),
  \eeq
  where $\gamma_{i,\nu} \neq 0$.
  At $\nu=\nu_c$,  the nature of the singularity changes:
  \beq\label{sing_crit}
  T_i= \alpha_{\nu_c} +\beta_{\nu_c}(1-w/\rho_{\nu_c})+\gamma_{\nu_c} (1-w/\rho_{\nu_c})^{6/5}\,(1+o(1)).
  \eeq
  In asymptotic terms,
  \beq\label{asympt}
  [w^n]T_i\sim
  \begin{cases}
    \kappa_{i,\nu}\, \rho_\nu^{-n}  n^{-5/2}&\hbox{ for  } \nu\not=\nu_c,
    \\
    \kappa_{i,\nu_c}\, \rho_{\nu_c}^{-n}  n^{-11/5}&\hbox{ for  } \nu=\nu_c,
  \end{cases}
  \eeq
  where $\kappa_{i,\nu} >0$.
   For $\nu=0$ and $i>1$ the series $T_i$ has a unique singularity at $\rho_0=1/8$, with a planar map singular behaviour~\eqref{sing_generic}, while $T_1=0$ when $\nu=0$.  
\end{Proposition}

\begin{figure}[htb]
  \centering
  \includegraphics[width=60mm]{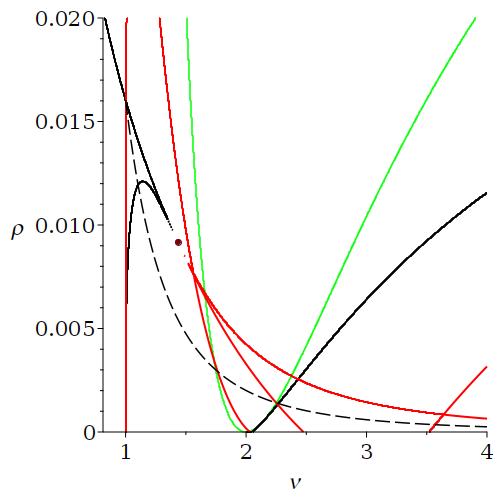} 
  \hskip 10mm   \includegraphics[width=60mm]{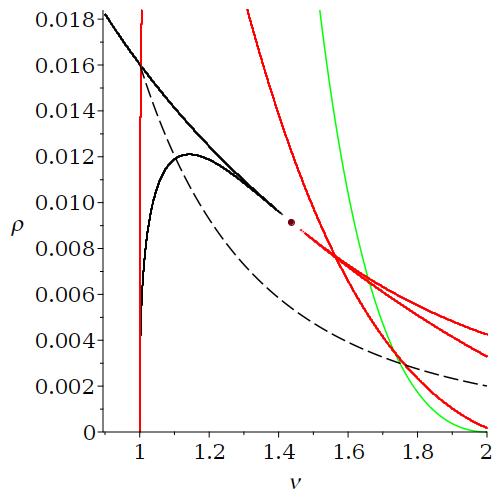} 
  \caption{The branches of $\Delta_0$ (green/light), $\Delta_1$ (black) and $\Delta_2$ (red). The black dashed curve is the lower bound $\rho_1/\nu^3$ on the radius, for $\nu\ge 1$. The plot on the right  zooms on the interval $[1,2]$. The radius $\rho_\nu$ first follows the top black branch, between $\nu=0$ and $\nu_c\simeq 1.44$, and then the top red branch.}
  \label{fig:radius}  
\end{figure}

\noindent

\noindent{\bf Remarks}\\
1. The above result can be compared to the analogous result for the Ising model on triangulations (the  Potts model with $2$ colours only), where the critical value of $\nu$ is at $1+1/\sqrt 7$, with exponent~$4/3$ rather than $6/5$ in the singular expansion of the series at criticality;
see~\cite[Claim~24]{BeBM-11} or~\cite[Thm.~2.4]{albenque-Ising}.\\
2. The exponent $-11/5$ occurring in~\eqref{asympt} is in agreement with the prediction given by the Knizhnik--Polyakov--Zamolodchikov (KPZ) formula~\cite{KPZ}: the $3$-state Potts model having \emm central charge, $c=4/5$ (see, e.g.,~\cite[Eq.~(4.22)]{Jacobsen-saleur-Potts}), the KPZ formula gives, at criticality, an asymptotic estimate  for the  Potts-weighted number of (rooted) maps of size $n$, of the form
\[
  \kappa \mu^n n^{\gamma-2},
\]
where the \emm string susceptibility exponent, $\gamma$ is
\[
  \gamma= \frac 1 {12}\left( c-1-\sqrt{(1-c)(25-c)}\right) = -\frac 1 5,
\]
so that $\gamma-2=-11/5$ indeed. For the Ising model, $c=1/2$ hence $\gamma=-1/3$, in agreement with the exponent $-1-4/3=-7/3$ found in the number of Ising-weighted maps of size $n$.  

\medskip
The details of the proof require some care, which makes the proof  long. We have thus split it into several shorter lemmas.
As many parts of this paper, this proof requires using a computer algebra system. Our \Maple\ session is available on our web pages. Inside \Maple, we use the packages {\tt algcurves, plots}, {\tt gfun}~\cite{gfun} and {\tt DA}~\cite{DiffApprox}. We also use  {\tt msolve}, a C library for solving in arbitrary precision multivariate polynomial systems~\cite{msolve}.

\begin{Lemma}[\bf The case $\boldsymbol{\nu=0}$]
  \label{lem:0}
  Proposition~\ref{prop:asympt} holds true for $\nu=0$: all series $T_i$ with $i>1$ have radius of convergence $\rho_0=1/8$. This is their unique
  singularity, and they all have a map-like singular expansion of the form~\eqref{sing_generic} near $\rho_0$.

  For $\nu>0$, the radius of convergence of $T_1$ satisfies $\rho_\nu \le \rho_0=1/8$.
\end{Lemma}
\begin{proof}
  When $\nu=0$ we have $T_1=0$, so we first focus on the series $T_2$ that counts properly $3$-coloured
  near-triangulations of outer degree $2$. Since in general $T_2=T_1/\nu$,   we can derive the minimal polynomial of $T_2$ from that of $T_1$, computed in Section~\ref{sec:T1}. When $\nu=0$ we find that~$T_2$ is quadratic only:
  \beq\label{min:T2nu0}
    8 T_2^{2}-\left(8 w^{2}+12 w -1\right) T_2 +2 w^{2} \left(w^{2}+11 w -1\right) =0   .
  \eeq
 This gives:
  \[
    T_2(0,w)=\frac 1 {16}\left( (1-8w)^{3/2}-1+12w+8w^2\right),
  \]
with  radius $\rho_0=1/8$.  There is no other singularity. Near $w=\rho_0$ the singular behaviour of $T_2$ is of the map-type~\eqref{sing_generic}. 

  Let us now discuss the series $T_i$ for $i>2$, still with $\nu=0$. First, by deleting the root edge in a near-triangulation of outer degree $2$, we obtain that $T_2=2 w^2+T_3$ at $\nu=0$. In particular, $T_3$ has the same singularity and the same singular behaviour as $T_2$. At this point we have reached Tutte's classical result on \emm bicubic maps,, as discussed in the introduction; see~\cite[p.~269]{tutte-census-maps}. We then return to the 1-catalytic equation~\eqref{eqinv}, where we replace $T_1$ by $\nu T_2$:
  a factor $\nu$ comes out. After dividing by $\nu$, we set $\nu=0$ and $T_3=T_2-2w^2$. This gives the following equation:
  \begin{multline*}
    4 y^{5}T(y)^{3} -y^{2} \left(8 y^{2}+10 y -1\right) T(y)^{2}
    +2\left(3 w \,y^{3}+2 y^{2}+y -1\right) T(y)\\
    +2  \,y^{2} T_2  -w^{2} y^{2}
    + 2(y + 1)(1-2y )w
    =0.
  \end{multline*}
  Let us now define $S(y)$ by $T(y)=w+y^2T_2+y^3 S(y)$. Then the above equation yields
  \[
    S(y)= T_2-2w^2+ y \Pol(w, T_2, y, S(y)),
  \]
  for some polynomial $\Pol$ with coefficients in $\qs$. This gives by induction on $i\ge 3$ an expression of $T_i=[y^{i-3}]S(y)$ as a polynomial in $\qs[w,T_2]$. More precisely, in sight of~\eqref{min:T2nu0}, $T_i$ belongs to $\qs[w] + T_2 \qs[w]$. In particular, either $T_i$ is a polynomial in $w$, or it is an algebraic (quadratic) function of $w$ with a unique singularity at $\rho_0=1/8$.
  It is easy to see, by adding a dangling edge in the outer face of a near-triangulation, that $T_{i+2}\ge w T_i$, coefficientwise. Hence $T_{2i} \ge w^{i-1}T_2$ and $T_{2i+1}\ge w^{i-1}T_3$. Given that neither $T_2$ nor $T_3$ is a polynomial in $w$, this proves that none of the $T_i$, for $i\ge 2$, is a polynomial in $w$. Since $T_i$ is a polynomial in $w$ and $T_2$, its singular behaviour near $w=1/8$ is in $(1-8w)^{3/2+k}$ for $k$ a non-negative integer. This implies that the coefficient of $w^n$ in $T_i$ grows like $8^n n^{-5/2-k}$. However, the above lower bounds on $T_{2i}$ and $T_{2i+1}$ then imply that $k=0$, so that $T_i$ has  a map-type singularity for any $i\ge 2$.

  Finally, we note that $T_1(\nu,w) \ge \nu T_2(0,w)$ coefficientwise, so that $\rho_\nu\le \rho_0=1/8$ for $\nu>0$.
\end{proof}

We next focus on the series $T_1$, and will return to the series $T_i$ for $i>1$ in Lemma~\ref{lem:Ti}. 
We refer to~\cite[Chap.~VII.7]{flajolet-sedgewick} for generalities on  singularities of algebraic series. In particular, given an annihilating polynomial of a series $F(w)$, say $\Pol(F(w))=0$ with coefficients in $\qs[w]$, all singularities of $F$ are found among the roots of the leading coefficient and of the discriminant of $\Pol$. Also, for series with non-negative coefficients, the radius of convergence is one of the singularities (Pringsheim's theorem).

\begin{Lemma}[\bf The case $\boldsymbol{\nu=1}$]
  \label{lem:1}
  Proposition~\ref{prop:asympt} holds true for $\nu=1$ and $i=1$: the series $T_1$ has radius of convergence $\rho_1=\sqrt3/108$, which is a root of $\Delta_1(1, \cdot)$. This is the unique  singularity, and $T_1$ has a map-like singular expansion of the form~\eqref{sing_generic} near $\rho_1$.

  For $\nu\ge 1$, the radius of convergence  of $T_1$ satisfies $\rho_\nu \ge \rho_1/\nu^3$.
\end{Lemma}
\begin{proof}
  As recalled in the introduction, this case is simple, and equivalent to the classical enumeration of near-triangulations counted by vertices (with a weight $3w$ per vertex)~\cite{mullin-nemeth-schellenberg}.   When $\nu=1$  the minimal polynomial of $T_1$, determined in Section~\ref{sec:T1}, factors. Choosing the correct factor yields:  
  \[
    192 T_1^{3}-\left(288 w -1\right) T_1^{2}+w \left(90 w -1\right) T_1 -3 w^{3} \left(81 w -1\right)=0  .
  \]
  The leading coefficient is non-zero, and the discriminant has only two non-zero roots $\rho_1=\sqrt3/108$ and $-\rho_1$, so that $\rho_1$ is the radius of convergence.  What about singularities at $-\rho_1$? Expanding the three branches of the above equation at $-\rho_1$ reveals that two are singular, and take a negative value, while the third is analytic and positive (such expansions can be computed using for instance the {\tt algeqtoseries} command in the {\tt gfun} package of \Maple). A plot of all three branches as a function of $w$ shows that $T_1$ is positive at~$-\rho_1$, and is thus the analytic branch. Finally, a local expansion of $T_1$ near~$\rho_1$ yields an expansion of the planar map type~\eqref{sing_generic}.

  We can now derive
  a lower bound on the radius of $T_1$ for $\nu\ge1$: since a near-triangulation of outer degree $1$ having $n$ vertices has $3n-4$ edges, each of them getting a weight at most $\nu$, we have $T_1(\nu,w) \le \nu^{-4}T_1(1, w\nu^3)$ coefficientwise, which implies $\rho_\nu \ge \rho_1/\nu^3$
  for $\nu\ge1$.
\end{proof}

We now return to general values of $\nu$.
We want to determine  the dominant singularities of $T_1$, and its  behaviour near these values.
It is sufficient to study the singular behaviour of $\DT=\partial_w T_1$, which has a simpler minimal polynomial~\eqref{alg:dT1}, and then integrate.

\begin{Lemma}[\bf Locating possible singularities]
  \label{lem:Delta}
  Take $\nu>0$.
  Let $\Delta_1$ and $\Delta_2$ be defined as in Proposition~\ref{prop:asympt}. Moreover, let $\Delta_0= 16 \nu  \left(\nu -1\right) w -\left(\nu -2\right)^{2}$. Then any singularity of $T_1$ is a root of $\Delta_0 \Delta_1 \Delta_2$.  
\end{Lemma}
\begin{proof}
  For   $\nu=1$, we refer to Lemma~\ref{lem:1}.

  We now assume $\nu \neq 0, 1$ and return to  the minimal polynomial~\eqref{alg:dT1} of $\DT$, say $\Pol(\nu,w, z)$, of degree $11$ in $z$. Its leading coefficient $\nu^7$ does not vanish, so all singularities of $\DT$ will be found among the roots of the discriminant of $\Pol$ with respect to $z$. This discriminant reads:
  \[
    \kappa\, \nu^{45}(\nu-1)^{18} \Delta_0^2\Delta_1\Delta_2\Delta_3^2,
  \]
  where $\kappa\in \zs$, $\Delta_0$,
  $\Delta_1$ and $\Delta_2$ are as above,
  and  $\Delta_3$ is another polynomial in $\nu$ and $w$, of degree~$5$ in $w$.
  We  examine similarly the leading coefficient and the discriminant of the minimal polynomial of $T_1$, and observe that they do not contain the factor $\Delta_3$. Hence the roots of $\Delta_3$ cannot be singularities of $T_1$ (unless they are also roots of $\Delta_0\Delta_1\Delta_2$).
\end{proof}

Let us now recall a classical lemma. The arguments of the proof can be found for instance in~\cite{HTW82}. 
  \begin{Lemma}
    Let $F\equiv F(\nu,w)$ be a power series in $w$ with coefficients in $\rs_{\ge 0} [\nu]$. Assume  that the degree in $\nu$ of the coefficient of $w^n$ in $F$ grows at most linearly in $n$. Then the radius of convergence $\rho_\nu$ of $F(\nu, \cdot )$ is a continuous non-increasing function of $\nu$ on $(0, +\infty)$, taking its values in $[0, +\infty]$. Moreover, $\log \rho_\nu$ is a concave function of $\log \nu$.
  \end{Lemma}
The final property is here for the sake of completeness, but we will not use it.

\begin{Lemma}[\bf The radius of convergence]
  \label{lem:radius}
  Let $\nu> 0$.  The radius of convergence of $T_1$, denoted  $\rho_\nu$, is a continuous non-increasing function of $\nu$, whose value is given by Proposition~\ref{prop:asympt}.
\end{Lemma}
\begin{proof}
  The coefficient of $w^n$ in $T_1$ is a polynomial of $\ns[\nu]$ of degree at most $3n-4$ (this is the number of edges in a map of $\mathcal T_1$ having $n$ vertices). By the above lemma, $\rho_\nu$ is a  continuous, non-increasing, log-convex function of $\nu$ on $(0, +\infty)$.
  The graph of this function is obtained by gluing parts of branches of the above polynomials $\Delta_i$, for $0\le i\le 2$.
  We refer to Figure~\ref{fig:radius} for plots of real positive branches.
  In what follows, we often replace arguments based on estimates of branches of algebraic functions at some point in controlled precision, as should be done (see~\cite[Chap.~VII.7]{flajolet-sedgewick}, \cite{chabaud}), by discussions on \Maple\  plots of algebraic curves. We hope that our readers will find them convincing enough.

  The only branch of $\Delta:=\Delta_0 \Delta_1\Delta_2$ containing the point $(\nu,w)=(1,\rho_1)=(1, \sqrt3/108)$  is the branch of $\Delta_1$ that decreases from $\rho_0=1/8$ to $\rho_1$ as $\nu$ increases from $0$ to $1$ (Figure~\ref{fig:radius}). Let us denote this branch by $\mathcal B_1$.
    We will prove that on the interval $[0,1]$, this branch does not meet any other branch of $\Delta$ (the intersection with the red branch will be shown to occur for $\nu>1$).
  An intersection point $(\nu,w) \in \mathcal B_1$ would cancel $\Delta_1$ of course, and  either $\Delta_0 \Delta_2$, or $\partial_w \Delta_1$ (if two branches of $\Delta_1$ meet at this point). We confirm (rigorously) with {\tt msolve} that no such point exists with $\nu\in(0,1]$ and $w\in [\rho_1, 1/8]$.
  By continuity, we conclude that $\mathcal B_1$ gives the radius of convergence for $\nu\in (0,1]$.

  \medskip
  \paragraph{\bf Following branches}
  We have just proved that the radius follows the branch $\mathcal B_1$ between $\nu=0$ and $\nu=1$ (the \emm antiferromagnetic phase,, in physics terms). This goes on in a neighbourhood of $\nu=1$  as $\nu$ increases. Later the radius may follow another branch $\mathcal B_2$ of $\Delta$ at a point where  $\mathcal B_1$ and  $\mathcal B_2$  meet.
  We are thus led to determine all points  $(\nu, w)$ where several positive branches of $\Delta$ meet. At these points, either one of the~$\Delta_i$'s, seen as a polynomial in $w$, has a multiple root, or two of the $\Delta_i$'s vanish. We determine controlled approximations of these points using {\tt msolve}.
  We naturally restrict our attention to values  of  $(\nu,w)$ such that $\nu \ge 1$ and $w\in [0, \rho_1]$. Moreover, in Table~\ref{tab:intersect} we have also excluded values such that $w<\rho_1/\nu^3$ (Lemma~\ref{lem:1}). The $10$ points listed on this table can be seen on the plots of Figure~\ref{fig:radius}.

  \begin{table}[htb]
    \centering

    \begin{tabular}[h]{c|c|c|c}
      $i\backslash j$  & 0 & 1 & 2
      \\  \hline
      \multirow{3}{1em}{0}
                       & \hskip 30mm \ &  & $(1.6577,0.0067)$ \\
                       & $\emptyset$ & $\emptyset$ & $(1.6614,0.0065)$\\
                       & & & $(1.7342,0.0034)$\\
                       & & & $(2.3765,0.0027)$
      \\ \hline
      \multirow{3}{1em}{1}  & &           & $(1.0035,0.0159)$ \\
                       & & $(\nu_c, \rho_{\nu_c})\simeq(1.4375,0.0091)$ & $(\nu_c, \rho_{\nu_c})$\\
                       & & & $(2.4319,0.0025)$ \\
      \hline
      \multirow{4}{1em}{2}  & & &$(\nu_c, \rho_{\nu_c})$ \\
                       & & & $(1.5629,0.0076)$\\
                       & & & $(1.5653,0.0075)$\\
                       & & & $(3.6393,0.0008)$\\
      \hline
    \end{tabular}
    \vskip 2mm
    \caption{Intersection points $(\nu,w)$ of branches of $\Delta_i$ and $\Delta_j$, for $0 \le i\le j \le 2$. }
    \label{tab:intersect}
  \end{table}

  Based on this inspection of intersection points of branches, we now return to the radius of convergence of $T_1$. 
  As~$\nu$ increases away from $1$, the first branch that $\mathcal B_1$ meets is a locally increasing branch of $\Delta_2$, at $\nu\simeq 1.0035$. We ignore it because the radius is non-increasing.
  The next intersection is at 
  the value $\nu_c:=1+3/\sqrt{47}$ introduced in the statement of the proposition. There the value of $\mathcal B_1$ is found to be the number $\rho_{\nu_c}$ given by~\eqref{rhoc}. At the point $(\nu_c, \rho_{\nu_c})$, four (real) branches of $\Delta$ meet: two branches of $\Delta_1$ (in black in Figure~\ref{fig:radius}), namely   $\mathcal B_1$ and a lower branch that exists for $\nu<\nu_c$,
  and two branches of $\Delta_2$ (in red) that start at $\nu_c$ and proceed for $\nu>\nu_c$.
  One is higher than the other: a local expansion reveals that they differ by a sign in  the coefficient of $(1-w/\rho_{\nu_c})^{3/2}$. A similar statement holds for the two branches of $\Delta_1$ at $\nu_c$.
  The geometry of the branches then implies that $\rho_\nu$ is on the top branch of $\Delta_2$  above $\nu_c$ (Figure~\ref{fig:radius}): indeed, all non-increasing branches intersecting this branch for some $\nu\ge \nu_c$ reach the value $0$ at some point.
  This concludes the determination of $\rho_\nu$. We refer to our \Maple\ session for details.
\end{proof}

\begin{Lemma}[\bf Nature of the dominant singularity]
  \label{lem:sq-root}
  Let $\nu>0$. The singular behaviour of~$T_1$ near its radius $\rho_\nu$ is given by~\eqref{sing_generic} if $\nu\neq \nu_c$, and by~\eqref{sing_crit} if $\nu=\nu_c$.
\end{Lemma}
\begin{proof}
   We use again the {\tt gfun} package in \Maple, and more precisely the {\tt algeqtoseries} command~\cite{gfun}, which implements the Newton polygon method and computes Puiseux expansions at a given point $w_0$ of all roots of a polynomial with coefficients in $\qs(w)$ (or $\qs(\nu,w)$). See for instance~\cite[Sec.~VII.7.1]{flajolet-sedgewick} or~\cite[Sec.~6.1]{stanley-vol2} for generalities on these expansions. We apply this procedure to the minimal polynomial of $\DT$, given by~\eqref{alg:dT1}.

  $\bullet$ Let us begin with the critical case  $\nu=\nu_c$. At this point we find $6$ roots with distinct constant terms, plus $5$ other roots that have the same constant term. The only real one expands as
  \[
    \frac{25}{38} - \frac{\sqrt{47}}{19}
    - \left(\frac 7{2^4 19^5}\left( 8555\sqrt{47}- 57585\right)\right)^{1/5}\left(1-w/\rho_{\nu_c}\right)^{1/5} + \LandauO\left((1-w/\rho_{\nu_c})^{2/5}\right),
  \]
  and the remaining $4$ are obtained by multiplying the second coefficient by a non-trivial fifth root of unity. This implies that the first $6$ roots are analytic at $\rho_{\nu_c}$, and that the above expansion is that of $\DT$ (because $\DT$ \emm is, singular at its radius). By integration, this yields~\eqref{sing_crit}.

  $\bullet$ We go on with $\nu<\nu_c$, in which case $\rho_\nu$ is a root of $\Delta_1$. In our analysis we first consider~$\nu$ as an indeterminate.
  We first need to  determine the possible values of $\DT$ at the point $w=\rho_\nu$. We take the minimal polynomial~\eqref{alg:dT1} of $\DT$, say $\Pol(\nu,w,z)$, specialize it at $w=\rho_\nu$ and factor it over $\qs(\nu, \rho_\nu)$.  We find two irreducible factors: the first one, say $\Pol_1(\nu,\rho_\nu,z)$, has degree $9$ in~$z$, with (generically) $9$ distinct roots;
  the other one is the square of a polynomial  of degree $1$ in $z$, with  an explicit (double) root $c_\nu \in \qs(\nu)[\rho_\nu]$. Since $\DT(\rho)$ must be a multiple root of $\Pol(\nu, \rho, \cdot)$, we conclude that $\DT(\rho)=c_\nu$, generically. By elimination of $\rho_\nu$ in the expression of $c_\nu$, we find its minimal polynomial:
  \begin{multline}\label{cnu_min}
    20736 \nu^{4} c_\nu^{5}-432 \nu^{3} \left(77 \nu +43\right) c_\nu^{4}+48 \nu^{2} \left(451 \nu^{2}+484 \nu +109\right) c_\nu^{3}\\
    -36 \nu  \left(203 \nu^{3}+293 \nu^{2}+143 \nu +9\right) c_\nu^{2}+4 \left(317 \nu^{4}+559 \nu^{3}+363 \nu^{2}+55 \nu +2\right) c_\nu \\
    -\nu  \left(89 \nu^{3}+189 \nu^{2}+147 \nu +7\right)=0.
  \end{multline}
  We now expand the solutions  (in $z$) of $\Pol(\nu,w,z)$  in the vicinity of $(w,z)=(\rho_\nu,c_\nu)$ in powers of $1-w/\rho_\nu$, and find, using again the {\tt algeqtoseries} command, two series with a square root singularity (in the generic case).
  The non-decreasing branch must be $\DT$. We thus obtain: 
  \beq\label{DT-sing}
  \DT= c_\nu -(d_\nu)^{1/2} (1-w/\rho_\nu)^{1/2} + \LandauO (1-w/\rho_\nu),
  \eeq
  where $d_\nu\in \qs(\nu, \rho_\nu)$ is explicit.  The other singular branch is obtained by changing the sign in the second term.  This gives~\eqref{sing_generic} by integration, and this holds for any $\nu\in(0, \nu_c)$ under the following two conditions:
  \begin{itemize}
  \item[--] the polynomial
    $\Pol(\nu,\rho_\nu,z)$ has indeed $10$ distinct roots, with $c_\nu$ as its unique double root,
  \item[--] the above number $d_\nu$ is well-defined and non-zero.
  \end{itemize}
  These conditions hold except for finitely many (algebraic) values of $\nu$.

  We will now prove that the singular behaviour~\eqref{DT-sing} holds in fact in the whole interval $(0, \nu_c)$. First,  the value of $\DT$ at its radius, namely  $\DT(\nu,\rho_\nu)$, is continuous at every point $\nu$ as a function taking its values in $\rs \cup \{\infty\}$ (because the coefficients of $\DT$ are non-negative). Then, we observe that when $\nu \rightarrow 0^+$, only one solution of~\eqref{cnu_min} is real: this solution must coincide with  $c_\nu$ near~$0$ (except possibly at finitely many values of $\nu$).
  Its expansion near $0$ is $c_\nu=7\nu/9+\LandauO(\nu^2)$. Moreover, this solution 
  is continuous and increasing on $(0, \nu_c)$, reaching at $\nu_c$ the value $\DT(\nu_c, \rho_{\nu_c}) = \frac{25}{38} - \frac{\sqrt{47}}{19}$ (Figure~\ref{fig:sing_val_DT1}).
  By continuity,
  $c_\nu$ coincides with  $\DT(\rho_\nu)$ on the whole interval $(0, \nu_c)$, including at possibly non-generic points. So we just have to expand the roots (in $z$) of $\Pol(\nu, w,z)$ near $w=\rho_\nu, z= c_\nu$. This is what we have done in~\eqref{DT-sing}, for~$\nu$ generic. But now we examine the Newton polygon procedure step by step, to see what could go wrong for certain specific values of $\nu$. This is inspired from~\cite[Prop.~3.4]{Bernardi-Curien-Miermont} and~\cite[Lem.~2.7]{albenque-Ising}.

  Recall that $c_\nu$ belongs to $\qs(\nu)[ \rho_\nu]$.  Starting from the minimal polynomial~\eqref{alg:dT1} of $\DT$, we form by elimination of $\rho_\nu$ (which is a root of $\Delta_1$) a polynomial $\overline\Pol ( Y,\vareps)$ with coefficients in $\qs[\nu]$ that vanishes for $\varepsilon:= 1-w/\rho_\nu$ and $Y:=\DT(\nu,w)-\DT(\nu,\rho_\nu)=\DT -c_\nu$. This calculation is done for a generic value of $\nu$. However, by continuity in $\nu$, this polynomial vanishes at the above values of $\vareps$ and $Y$ for any $\nu\in (0, \nu_c)$. It has degree $55$ in $Y$, and, by construction, vanishes when $Y=\varepsilon=0$. We now apply to it the Newton polygon method; see, e.g.,~\cite[Sec.~VII.7.1]{flajolet-sedgewick}. All monomials $Y^i\vareps^j $ that occur in $\overline\Pol$  satisfy $i+2j \ge 10$, and, generically, the coefficients of $Y^0 \vareps^5 $ and $Y^{10}\vareps^0$ are non-zero: the only negative slope in the Newton polygon being  $-1/2$, this explains the square root behaviour found above.
  We next examine for which values of $\nu\in (0,\nu_c)$ one of these coefficients (or both)  vanish: we only find two suspicious values, namely $\nu=1$ (for which we know that~$\DT$ has a square root singularity, see Lemma~\ref{lem:1}), and $\nu=   \nu_b:=  1-3/\sqrt{47}$, the conjugate of $\nu_c$.

  So it remains to study this case. Returning to the value of $\Delta_1$, we find that $\rho_\nu$ is cubic over $\qs(\sqrt{47})$.  The same holds for  $\DT(\rho_\nu)$.
  For this value of $\nu$, the polynomial $\overline \Pol(Y, \vareps)$ factors over $\qs(\sqrt{47})$ as $\overline \Pol_1 (Y, \vareps)\overline \Pol_2(Y,\vareps)$, with $\overline \Pol_1$ (resp. $\overline \Pol_2$) of degree $33$ (resp. 22) in $Y$. Applying the Newton polygon method shows that the solutions of $\overline \Pol_1$ that vanish at $\vareps=0$ will have a square root singularity, but those of $\overline \Pol_2$ will have a singularity in $\vareps^{1/5}$. So it remains to check that $\overline \Pol_2( \DT(w) - \DT(\rho_\nu),1-w/\rho_\nu) \neq 0$ for $\nu=\nu_b$. To do this we evaluate the above expression at $w=0$:  and indeed, $\overline \Pol_2( - \DT(\rho_\nu),1)$ does not reduce to $0$ modulo the minimal polynomial (over $\qs(\sqrt{47})$) of $\DT(\rho_\nu)$.

  \begin{figure}[htb]
    \centering
    \includegraphics[width=60mm]{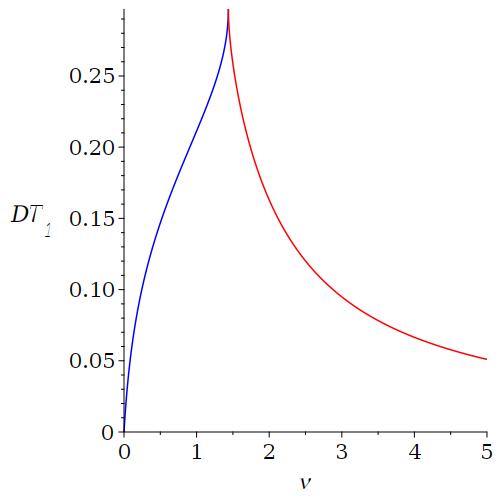}
    \hskip 10mm
    \caption{The value $c_\nu$ of $\DT$ at its radius $\rho_\nu$ increases up to $\nu_c$ and decreases afterwards.
    }
    \label{fig:sing_val_DT1}  
  \end{figure} 

   $\bullet$ We finally address the case $\nu>\nu_c$, where $\rho_\nu$ is a root of $\Delta_2$. The analysis parallels completely the case $\nu<\nu_c$. This time the value $c_\nu$ of $\DT(\nu,\rho_\nu)$ at its radius of convergence is algebraic of degree $9$ rather than $5$ (as $\rho_\nu$ itself), and is a decreasing function of $\nu$ (Figure~\ref{fig:sing_val_DT1}). Computations are heavier in this case  because the radius $\rho_\nu$ has now degree $9$ instead of $5$.  When applying the Newton polygon method to prove the square root behaviour, we obtain a polynomial $\overline\Pol(Y, \vareps)$ of degree $99$ in $Y$, with coefficients in $\qs[\nu]$, that vanishes when $\vareps=1-w/\rho_\nu$,  
   and $Y=\DT  -\DT(\rho_\nu)$.  All monomials $ Y^i\vareps^j$ that occur in it satisfy $i+2j \ge 18$, and, generically, the coefficients of $Y^0\vareps^9 $   and of $Y^{18}\vareps^0$ are non-zero: this proves the square root behaviour, except for  seven values of $\nu$ for which one of these two coefficient vanishes (or both). 
 
 These values are $\nu=2$, $\nu=1+3/\sqrt 2$, one value of degree $6$ over $\qs$, three (conjugate) values of degree $10$, and finally one of degree $16$. For the first six values, the specialized polynomial $\overline\Pol(Y, \vareps)$ (or each of its factors) has a unique negative slope $-1/2$ in its Newton polygon again. We thus conclude to a square root singularity in $\DT$. 
  We refer to our \Maple\ session for details.  

  The final value of $\nu$, of degree $16$, is more tricky as one segment in the polygon has slope~$-1$.
  This is analogous to the difficulty raised by $\nu=\nu_b$ in the case $\nu<\nu_c$, where one found the slopes~$-1/2$ and $-1/5$ in the Newton polygon. In principle, we could apply here the same strategy  as for $\nu=\nu_b$: for this value of $\nu$, denoted $\nu_{16}$, the polynomial $\Delta_2$ factors over~$\qs(\nu)$, and one finds $\rho_\nu$ and $\DT(\rho_\nu)$ to be of degree $8$ over $\qs(\nu)$. We expect $\overline\Pol$ to factor into a term of degree $88$ and one of degree $11$. But this factorization just did not finish on our laptops, and there is in this case a more theoretical argument, which we now explain. The relevant part of the set of points $(i,j)$ such that $ Y^i\vareps^j$ occurs in $\overline\Pol$, for $\nu=\nu_{16}$, is shown in Figure~\ref{fig:newton_nu16}.
  Analysing this diagram first tells us that exactly $19$ solutions $Y$ of $\overline \Pol(Y, \eps)=0$ vanish at $\eps=0$.  General results on the Newton polygon (see e.g.~\cite[Ch.~II, Prop.~6.3]{neukirch})
  imply that exactly one of them will start with a term of the order of $\eps$. This means in particular that this solution does not have a further branching, so it is analytic near $\eps=0$: this cannot be $\DT(w)-\DT(\rho_\nu)$, which we know to be singular. So $\DT(w)-\DT(\rho_\nu)$ must be one of  the $18$ other solutions, and all of them have a square root singularity. This concludes the proof of the lemma.
\end{proof}

  \begin{figure}[htb]
    \centering
    \includegraphics[width=55mm]{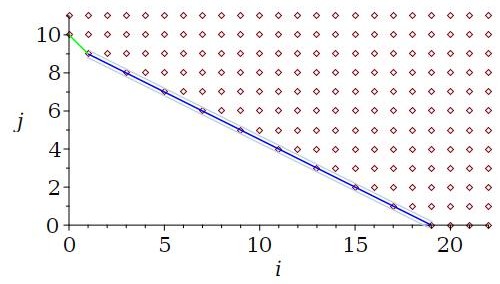}     
    \caption{The south-west part of the Newton polygon of $\overline\Pol(Y,\eps)$ for $\nu=\nu_{16}$.} 
    \label{fig:newton_nu16}
  \end{figure}

Let us now address the (non-)existence of other singularities of minimal modulus, called \emm dominant, singularities. With a parameter $\nu$, this is a difficult task, and we believe to provide new tools for it. {We actually give two proofs: the first one rules out the existence of multiple dominant singularities by examining the algebraic conditions and inequalities that $\nu$, $\rho_\nu$, $\DT(\rho_\nu)$ and other algebraic quantities should satisfy, and proving, using {\tt msolve}, that there is no solution. These ideas can be applied in many contexts. The second idea was suggested to us by Andrew Elvey Price. Roughly speaking, it says that a series that is obtained as the composition of two series with non-negative coefficients cannot have multiple dominant singularities. This idea is used in one of his recent papers, with Nessmann and Raschel~\cite[Lem.~10]{elvey-nessmann-raschel}. It gives a more combinatorial explanation of the uniqueness of the dominant singularity. One difference with~\cite{elvey-nessmann-raschel} is that in the latter paper, the series under study, say $U(z)$ is expressed as $V(\Rat(z))$ from some explicit rational function $\Rat(z)$, while the result that we use here relies on an implicit function schema, $U(z)=G(z,U(z))$ (see Proposition~\ref{lem:compo_unique}).}

\begin{Lemma}[\bf Uniqueness of the dominant singularity]
  \label{lem:unique}
  Let $\nu>0$. The series $T_1$ has a unique dominant singularity, which is its radius of convergence $\rho_\nu$. Hence the estimates~\eqref{asympt} hold for $i=1$.
\end{Lemma}
\begin{proof}[First proof of Lemma~\ref{lem:unique}]
  We begin by studying separately the cases $\nu=1$, $\nu=2$ and $\nu=\nu_c$. For $\nu=1$ the result is stated in Lemma~\ref{lem:1}. For $\nu=2$, we check (numerically) that the only root of $\Delta:=\Delta_0\Delta_1\Delta_2$ that has modulus $\rho_2$ is $\rho_2$ itself. The same holds for $\nu=\nu_c$. For instance, the $12$ roots of $\Delta$ for $\nu=\nu_c$ are plotted in Figure~\ref{fig:roots_disc2}, together with the circle of radius $\rho_{\nu_c}$. Observe that  $\Delta$ has generically $1+5+9=15$ roots, but at $\nu_c$ there is a root $\rho_{\nu_c}$ of multiplicity $4$.

  \begin{figure}[htb]
   \centering
    \includegraphics[width=55mm]{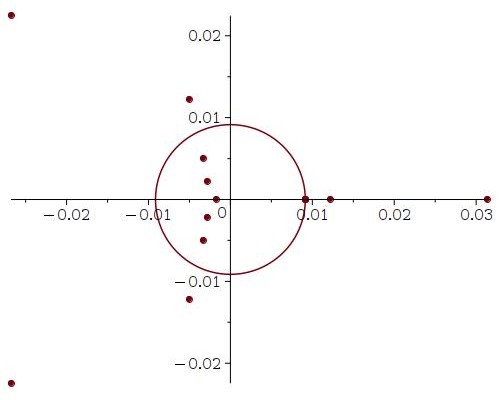}    
    \caption{The roots of $\Delta(\nu_c, \cdot)$.}
    \label{fig:roots_disc2}
  \end{figure}

  We now take $\nu>0$, with $\nu \not \in \{ 1, 2, \nu_c\}$. 
  Let $s=x+iy$, with $x$ and $y$ real, be a dominant singularity of $T_1$, or equivalently $\DT$, distinct from $\rho_\nu$. By Lemma~\ref{lem:Delta}, it satisfies $\Delta_j(\nu,s)=0$, for some $j\in \{0,1,2\}$. The series $\DT$ converges at this point (since it converges at $\rho_\nu$ by Lemma~\ref{lem:sq-root}), and, since it is an aperiodic series with non-negative coefficients, we have $\left|\DT(\nu, s)\right|< \DT(\nu, \rho_\nu)$.

  Let us again denote by $\Pol(\nu,w,z)$ the minimal polynomial of $\DT$, see~\eqref{alg:dT1}. Since $s$ is a singularity of $\DT$, the value $z=\DT(\nu,s)$ must be a multiple root of $\Pol(\nu,w,z)$. Hence we have the following system of $6$  equations relating the $5$ values $\nu$, $\rho\equiv\rho_\nu$, $s$, $\DT(\rho)\equiv \DT(\nu, \rho)$, and $\DT(s)\equiv \DT(\nu, s)$:
  \beq\label{syst}
  \begin{cases}
    \Delta_i(\nu, \rho)&=0,\\
    \DT( \rho) &= c_\nu,    \\
    |s|^2      &=\rho^2,
  \end{cases}
  \hskip 15mm          \begin{cases}
    \Delta_j(\nu,s)&=0,\\
    \Pol\big(\nu, s, \DT(s)\big)&= 0,\\
    \partial_z \Pol\big(\nu, s, \DT(s)\big)&=0,
  \end{cases}
  \eeq
  where $i=1$ if $\nu<\nu_c$ and $i=2$ otherwise, and $c_\nu$ is a fraction in $\nu$ and $\rho$, determined in the proof of Lemma~\ref{lem:sq-root}, which takes two different values depending on whether $\nu<\nu_c$ or $\nu>\nu_c$; see Figure~\ref{fig:sing_val_DT1}.

  If we look for \emm real, dominant singularities, corresponding to $s=-\rho$, the system simplifies to $5$ real equations in $4$ real unknowns:
\[ 
  \begin{cases}
    \Delta_i(\nu, \rho)&=0,\\
    \DT( \rho) &= c_\nu,  
  \end{cases}
  \hskip 15mm          \begin{cases}
    \Delta_j(\nu,-\rho)&=0,\\
    \Pol\big(\nu, -\rho, \DT(-\rho)\big)&= 0,\\
    \partial_z \Pol\big(\nu, -\rho, \DT(-\rho)\big)&=0.
  \end{cases}
  \] 
  This can be turned into a   \emm polynomial, system by taking the numerator of $\DT(\rho)-c_\nu$ rather than $\DT(\rho)-c_\nu$ itself.

  If we look for non-real dominant singularities, $s=x+iy$ with $y\neq 0$, each of the $3$ equations on the right-hand side of~\eqref{syst} splits into $2$ real equations, relating $x$, $y$, $U:=\Re(\DT(s))$ and $V:=\Im(\DT(s))$, giving a total of $9$ real equations for $7$ real unknowns. For instance, $\Delta_1(x+iy)=0$ with $y\neq0$ splits into two real polynomials in $x$ and $z:=y^2$, namely
  \[
    \Delta_1(x+iy)+\Delta_1(x-iy)=0, \qquad \text{and} \qquad \frac{ \Delta_1(x+iy)-\Delta_1(x-iy)=0}{iy}=0.
  \]

  The quantities $\nu$, $\rho$, $s$, $\DT(\rho)\equiv \DT(\nu, \rho)$ and $\DT(s)\equiv \DT(\nu,s)$ are moreover constrained by the following inequalities:
  \beq\label{ineq1}
  \left\{
    \begin{array}{ccl}
      0<\nu<\nu_c, &\quad  \rho_{\nu_c}< \rho<\rho_0 ,& \text{if } i=1,\\
      \nu_c<\nu, &\quad  \rho_1/\nu^3 <\rho< \rho_{\nu_c}, & \text{if } i=2,
    \end{array}    \right.
  \eeq
  and
  \beq\label{ineq2}
  0<   \DT(\rho)< \DT(\nu_c, \rho_{\nu_c}), \qquad
  \left| \DT(s)\right|< \DT(\rho),
  \eeq
  where
  \[
    \rho_0=1/8, \qquad \rho_1= \sqrt3/108, \qquad \rho_{\nu_c}= \frac{1295\sqrt{47}-7875}{109744},  \qquad \DT(\nu_c, \rho_{\nu_c}) = \frac{25}{38} - \frac{\sqrt{47}}{19}.
  \]

  We now take the $6$ real polynomial systems obtained for $(i,j)\in\{1,2\}\times \{0,1,2\}$, each system being declined in two versions, the real case $s=-\rho$ ($5$ equations) and the non-real case ($9$ equations). We use {\tt msolve} to approximate their real solutions (rigorously, up to arbitrary precision). We then examine which of these (finitely many) solutions satisfy~\eqref{ineq1} and~\eqref{ineq2}. In most cases, no solution remains.

  For instance, when $i=j=1$ and $s=-\rho$, {\tt msolve} returns $23$ candidates for $(\nu, \rho, \DT(\rho), \DT(-\rho))$. Only $7$ of them satisfy $0<\nu<\nu_c$. Among them, only $3$ satisfy  $\rho_{\nu_c}<\rho<1/8$. Each of these 3 would be such that $\left| \DT(s)\right|>\DT(\rho)$, and this case is solved.

  In some cases we use an additional sieve: we observe that the value found for $\rho$, supposed to be $\rho_\nu$, is not on the correct  branch of $\Delta_i$ (Figure~\ref{fig:radius}) and thus cannot be the radius. This allows us to exclude more points. For instance, when $i=2$ and $j=1$, {\tt msolve} returns in the non-real case $39$ solutions. The conditions on $\nu$ and $\rho$ allow us to restrict our attention to  $7$ of them.  We can exclude $4$ more because the value of $\rho$ is not on the correct branch of $\Delta_2$. For the remaining ones  we find that $\left| \DT(s)\right|>\DT(\rho)$.

  \smallskip
  We must mention two difficulties. First, when $j=2$, that is $\nu>\nu_c$, the denominator of~$c_\nu$ contains factors $(\nu-1)$ and $(\nu-2)$, and this seems to  make some of our systems positive dimensional.
  So when $j=2$, we add a new variable $c$ and complete our systems with an equation $c(\nu-1)(\nu-2)=1$ to remedy this problem.

  The other difficulty is due to the size of our systems, in the non-real case. 
  For the moment, we have actually only given to {\tt msolve} the equations that do not involve the series $\DT$. For each solution $(\nu, \rho, x, z)$ (with $z=y^2$), after applying the above sieves on $\nu$ and  $\rho$, and checked that~$\rho$ is on the correct branch of $\Delta_i$, we compute  (not certified) estimates of $\DT(\rho)$ and $\DT(x+iy)$ (using the final equations of the system) and  use the condition on the moduli of these values to rule out the remaining candidates.

  In the end, we conclude that the radius of convergence of $\DT$ is always the unique dominant singularity.
\end{proof}

  We now come to an alternative proof, which relies on the existence of a composition equation for the series $T_1$. Such equations are quite common in the world of maps~\cite{BaFlScSo-Airy}. The proof of the following proposition is given in Appendix~\ref{app:unique-compo}.

\begin{Proposition}\label{lem:compo_unique}
  Let $U(z)\in \rs[[z]]$ be an aperiodic\footnote{there exist three integers $i<j<k$ such that the corresponding powers of $z$ actually appear in $U(z)$ and $\gcd(j-i,k-i)=1$.} power series in $z$ that converges in a neighbourhood of $0$, and such that $U(0)=0$. Assume that 
  \beq\label{eq:UG}
    U(z)= G(z, U(z))
  \eeq
  for some series $G(z,u)=\sum_{m,n} g_{m,n} z^m u^n \in \rs[[z,u]]$ that converges in a neighbourhood of $(0,0)$ and satisfies the following conditions:
  \[
    g_{0,0}=0, \quad g_{m,n} \ge 0 \ \forall (m,n), \quad g_{0,1}< 1 \quad \text{and }\  \exists ( m,n) \text{ such that } g_{m,n}>0 \text{ with } n\ge 2 .
  \]
  Let $\mathcal R$ be the non-negative region of convergence of $G$:
  \[
    \mathcal R:= \left\{ (z,u)\in \rs_{\ge 0}^2 \text{ such that } G(z,u) <\infty\right\}.
  \]
  Assume moreover that the closure $\overline{\mathcal R}$ of $\mathcal R$ has no vertical boundary: that is, if $(z,u) \in \overline{\mathcal R} \cap \rs_{>0}^2$ and $0< u'<u$, then $(z,u') \in {\mathcal R}^ \circ$, the interior of $\mathcal R$.

  Then $U(z)$ has non-negative coefficients, its radius of convergence $\rho$ is finite, and  is the unique singularity of $U$ on its disk of convergence.
\end{Proposition}

\begin{proof}[Second proof of Lemma~\ref{lem:unique}.]
  We will first write a composition equation for the series $T_1$, or rather, for $T_1/w$. Recall that $T_1=\LandauO(w^2)$.
  
  Let $L(q; a, b)\equiv L(a,b)$
be the  \gf \ of \emm loopless, $q$-coloured near-triangulations of outer degree $2$, where $a$ (resp. $b$) records the number of monochromatic (resp. bicoloured) edges. The colour of the root vertex is prescribed, as always.

As already observed, we have $T_1=\nu T_2$. Given a near-triangulation $M$ of $\mathcal T_2$, two cases may occur:
\begin{itemize}
\item the boundary of the outer face of $  M$ consists of two loops; such maps are counted by $T_1^2/w$,
  \item the edges incident to the outer face are not loops. 
  \end{itemize}
  In the latter case, let us draw $M$ with the root face as the outer face, and consider a \emm maximal, loop, that is, a loop that is not included (for this  drawing) in another loop. Then the outer side of this loop belongs to the boundary of a finite face of degree $3$. The two other edges that bound this face share the same endpoints (Figure~\ref{fig:compo}). For every maximal loop, we then merge these two edges (and what lies between them) into a single edge. This yields a loopless near-triangulation of outer degree $2$, which we call the \emm projection, of $M$.

  \begin{figure}[htb]
    \centering
    \includegraphics{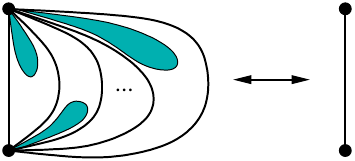}
    \caption{Compression of the maximal loops and the faces containing them.}
    \label{fig:compo} 
  \end{figure}

  Conversely, starting from  a loopless near-triangulation $\overline M$ of outer degree $2$, we obtain all maps of $\mathcal T_2$ that project on $\overline M$ by replacing each edge of $\overline M$  by a sequence of digons that include a map of $\mathcal T_1$ attached to one of the two corners of the digon (Figure~\ref{fig:compo}). If $\overline M$ has $m$ monochromatic edges and $p$ bicoloured ones, then the contribution in the series $T_2$ of maps that project on $\overline M$ is thus
  \[
  w^{5/3}   \left(    \frac {\nu w^{1/3}}{1-2\nu T_1/w}\right)^m
    \left(    \frac {w^{1/3}}{1-2 T_1/w}\right)^p
  \]
(recall that a near-triangulation of outer degree $2$ having $n$ vertices has $3n-5$ edges).  Putting together the above observations finally gives
  \[
    T_1= \nu T_2 = \nu \frac{T_1^2}w +
    w^{5/3} \nu  L  \left(    \frac {\nu w^{1/3}}{1-2\nu T_1/w},
      \frac {w^{1/3}}{1-2 T_1/w}\right).
  \]
  Note that $L(a,b)=\sum_{m,p} \ell_{m,p}a^m b^p$ where $p+m+2\equiv 0$ mod $3$ as soon as $\ell_{m,p}\neq 0$.

  We now apply Proposition~\ref{lem:compo_unique} with $z=w$, $U(z)=T_1(q, \nu,z)/z$ and
  \[
    G(z,u)= \nu u^2 + z^{2/3} \nu L\left(    \frac {\nu z^{1/3}}{1-2\nu u},
      \frac {z^{1/3}}{1-2 u}\right).
  \]
  The series $U(z)$ is aperiodic, as there exist coloured near-triangulations of outer degree $2$ with any number $n$ of vertices, for $n\ge 2$. The conditions on the coefficients of $G$ are easily checked. In particular, $g_{0,1}=0$. Moreover $G(z,u)$ converges around the origin as the number of near-triangulations with $n$ vertices grows exponentially.

  Now let $\cR_L$ be the non-negative region of convergence of $L$:
  \[
    \mathcal R_L:= \left\{ (a,b)\in \rs_{\ge 0}^2 \text{ such that } L(a,b) <\infty\right\}.
  \]
     Then $L$ is analytic in the interior of $\cR_L$. Moreover, if $(a,b) \in \overline \cR_L$ and $0<a'<a$, $0<b'<b$, then $(a',b')\in \cR_L^\circ$.
     Now, for $z, u\ge 0$, the series $G(z,u)$ converges if and only if
     \[
       u<\frac 1 2 \min(1, 1/\nu) \quad \text{and} \quad
       \left(    \frac {\nu z^{1/3}}{1-2\nu u},
         \frac {z^{1/3}}{1-2 u}\right) \in \cR_L.
     \]
     These conditions thus define the region $\cR$. Now assume that $(z,u)\in \overline \cR \cap \rs_{>0}^2$. Then
          \[
       u\le \frac 1 2 \min(1, 1/\nu) \quad \text{and} \quad
       \left(    \frac {\nu z^{1/3}}{1-2\nu u},
         \frac {z^{1/3}}{1-2 u}\right) \in \overline\cR_L.
     \]
     Given that $\nu>0$, then for $0<u'<u$ we have $(z,u')\in \cR^\circ$ by the above observation on $\cR_L$, hence $\overline\cR$ has no vertical boundary as required by Proposition~\ref{lem:compo_unique}. We conclude that $U(w)=T_1(w)/w$ has a unique dominant singularity.
\end{proof}

\begin{Lemma}[\bf The series $\boldsymbol{T_i}$ with $\boldsymbol{i>1}$]
  \label{lem:Ti}
  Let $\nu>0$.      For $i>1$, the series $T_i$ has radius of convergence $\rho_\nu$. This is its only dominant singularity. The singular behaviour of $T_i$ at this point is~\eqref{sing_generic} or~\eqref{sing_crit}, depending on whether $\nu\neq\nu_c$ or $\nu=\nu_c$.
\end{Lemma}
\begin{proof}
  We first establish  the following bounds, for $\nu>0$:
\[ 
  \nu^{2i-2} [w^{n-i+1}] T_1 \le  [w^n]T_i \le \frac 1 {\min (1,\nu)^{i-1}}  [w^n]T_1 . 
  \] 
  They  follow from two basic constructions, illustrated in Figure~\ref{fig:TiT1}. For the lower bound, we construct a near-triangulation of outer degree $i$ by adding $(i-1)$ vertices and $2(i-1)$ monochromatic edges to a near-triangulation of outer degree $1$. For the upper bound, we construct a near-triangulation of outer degree $1$ by adding $i-1$ edges, which may be monochromatic or not, to a near-triangulation of outer degree $i$.

  \begin{figure}[htb]
    \centering
    \scalebox{0.9}{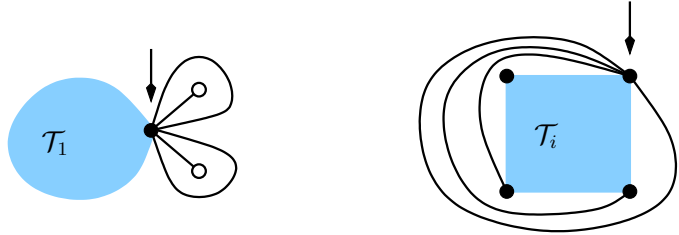}
    \caption{\emm Left., Construction of a map of $\mathcal T_i$ from a map of $\mathcal T_1$ (here, $i=3$).  \emm Right., Construction of a map of $\mathcal T_1$ from a map of $\mathcal T_i$ (here, $i=4$). }    
    \label{fig:TiT1}
  \end{figure}

  These bounds imply that $T_i$ has the same radius of convergence as $T_1$.
  Moreover, by Lemma~\ref{lem:unique}, its coefficients admit lower and upper bounds of the form~\eqref{asympt}: there exist positive constants $\kappa_1$ and $\kappa_2$, depending on $i$ and $\nu$, such that
  \beq\label{asympt:bounds}
  \kappa_1\, \rho_\nu^{-n}  n^{-1-e}\le [w^n]T_i \le  \kappa_2\, \rho_\nu^{-n}  n^{-1-e},
  \eeq
  with $e=3/2$  for $\nu\not=\nu_c$ and $e=6/5$ otherwise.

  We have proved in Section~\ref{sec:Ti} that $T_i$ belongs to $\qs(\nu,w,\DT)$, or equivalently to $\qs(\nu,w)[\DT]$.
  Thus $T_i$ is either rational in~$w$, or algebraic of degree $11$. The former case is impossible because of the  bounds~\eqref{asympt:bounds}, so~$T_i$ is algebraic of degree $11$. This is the asymptotic argument mentioned at the end of Section~\ref{sec:Ti}.

  Let us now discuss the dominant singularities of $T_i$. 
  We have already argued that $T_i$ has radius $\rho_\nu\equiv \rho$.  Assume that $T_i$ has  dominant singularities other than the radius. Since $T_1$ has a unique dominant singularity, and  $T_i\in\qs(\nu,w)[\DT]$,  all dominant singularities of $T_i$ distinct from~$\rho$  are poles. Moreover, if $s$ is a pole, its complex conjugate $\bar s$ is also one. Isolate these dominant poles by writing
  \[
    T_i= R(\nu,w) + S(\nu,w),
  \]
  where
  \begin{itemize}
  \item 
    $R(\nu,w)\in \rs(\nu,w)$ has only poles of modulus $\rho$,  distinct from $\rho$,
  \item  $S(\nu,w)\in \rs(\nu)[[w]]$ has $\rho$ as its unique dominant singularity.
  \end{itemize}
  Let $m\ge 1$ be the maximal order of a pole of $R(\nu,w)$. Then as $n$ tends to infinity,
  \[
    [w^n] R(\nu,w) = \rho^{-n} n^{m-1} \left( \alpha_1 \zeta_1^n+ \cdots + \alpha_k \zeta_k^n
      + \overline {\alpha_1} \overline {\zeta_1}^n+ \cdots + \overline {\alpha_k} \overline {\zeta_k}^n\right) + \LandauO (\rho ^{-n} n^{m-2}),
  \]
  for some complex numbers $\alpha_i$ and $\zeta_i\not=1$, with $|\zeta_i|=1$.
  Now by~\cite[Lem.~4]{Braverman}, there exists a constant $c>0$ such that, for $n$ large enough,
  \[
    [w^n] R(\nu,w) <- c \rho^{-n} n^{m-1}.
  \]
  Let $\beta \in \qs\setminus\{0, 1,2,\ldots\}$ be the exponent describing the singular behaviour of $T_i$, or equivalently $S(\nu, w)$,  near $w=\rho$. By this we  mean that $T_i$ differs from a function $H$ that is holomorphic at~$\rho$ by a term that is equivalent to $(1-w/\rho)^\beta$.
  Since $T_i$ converges at $\rho$ by~\eqref{asympt:bounds}, we have $\beta>0$. Thus for $n$ large enough,
  \[
    [w^n] T_i = [w^n] R(\nu,w) + [w^n]S(\nu,w) \le -c \rho^{-n} n^{m-1} + d \rho^{-n}n^{-1-\beta}, 
  \]
 with $m-1>-1-\beta$, contradicting the fact that $T_i$ has non-negative coefficients. We conclude that $T_i$ has a unique dominant singularity. The bounds~\eqref{asympt:bounds} finally imply that $\beta=3/2$ if $\nu\neq \nu_c$, and $\beta=6/5$ otherwise.
\end{proof}

\subsection{Cubic maps}
\label{sec:asympt-cubic}
We can study analogously the singularities of the 3-Potts \gf\ $K_i$ of near-cubic maps with root degree $i$, which are dual to near-triangulations of outer degree $i$. Recall that the corresponding Potts \gfs\ are related by~\eqref{eq:KT}, specialized at $q=3$.

\begin{figure}[b!]
  \centering
  \includegraphics[width=60mm]{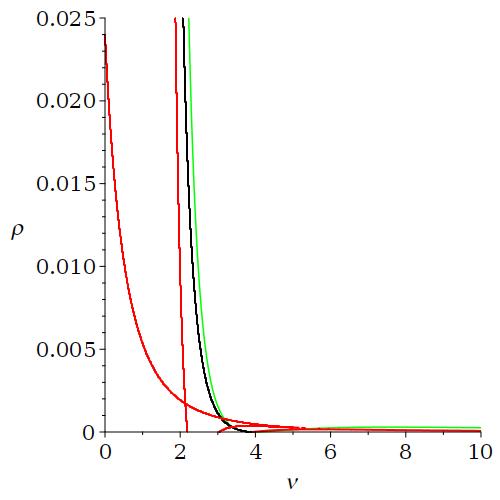} 
  \hskip 10mm   \includegraphics[width=60mm]{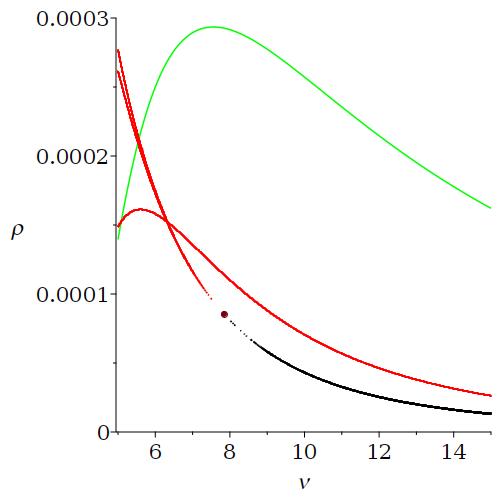} 
  \caption{The branches of $\tilde \Delta_0$ (green), $\tilde \Delta_1$ (black) and $\tilde \Delta_2$ (red).
    The plot on the right  zooms on the interval $\nu\in[5,15]$. The radius $\tilde \rho_\nu$ first follows the red branch, between $\nu=0$ and $\tilde \nu_c\simeq 7.85$, and then the top black branch (since two distinct black branches start at $\nu=\tilde \nu_c$, even if this is not clear on the figure).}
  \label{fig:radius-cubic}  
\end{figure}

\begin{Proposition}\label{prop:asympt-cubic}
  Let $\tilde\Delta_1$ and $\tilde\Delta_2$ be the  polynomials in $\nu$ and $\rho$ given by~\eqref{DeltaC1} and \eqref{DeltaC2} in Appendix~\ref{sec:app-cubic}.  Figure~\ref{fig:radius-cubic} shows, among other curves,  plots of
  the curves $\tilde \Delta_1(\nu,\rho)=0$ and $\tilde \Delta_2(\nu,\rho)=0$.  
  These polynomials are related to those of Proposition~\ref{prop:asympt} by
  \[
    \tilde \Delta_1(\nu, \rho)=\frac{(\nu-1)^{12}}{729} \Delta_1\left(\frac{\nu+2}{\nu-1}, \frac 1 3 (\nu-1)^3 \rho\right),
    \qquad
    \tilde \Delta_2(\nu, \rho)={(\nu-1)^{17}} \Delta_2\left(\frac{\nu+2}{\nu-1}, \frac 1 3 (\nu-1)^3 \rho\right).
  \]

  Let $i\ge 1$. Consider $K_i(\nu,w)\equiv K_i$ as a series in $w$ depending on a parameter $\nu\ge 0$. 
  Let $\tilde \rho_\nu$ denote its radius of
  convergence. Then $\tilde\rho_\nu$ is a continuous decreasing function of
  $\nu$ for $\nu\ge 0$, which satisfies
  \begin{align*}
    \tilde\Delta_2(\nu, \tilde\rho_\nu)&=0 \quad \hbox{for}\quad 0 \le \nu \le
                                         \tilde\nu_c:=1+\sqrt{47},\\
    \tilde\Delta_1(\nu, \tilde\rho_\nu)&=0 \quad \hbox{for} \quad\tilde\nu_c \le \nu.
  \end{align*} 
  More precisely, between $0$ and $\tilde\nu_c$ the radius $\tilde\rho_\nu$ is the branch of $\tilde\Delta_2(\nu, \rho)=0$ that starts at $\tilde\rho_0\simeq 0.024$ when $\nu=0$, and beyond $\tilde\nu_c$ the radius  is the highest of the two branches of  $\tilde\Delta_1(\nu, \rho)=0$ that start at
\[ 
  \tilde\rho_{\tilde\nu_c}=\frac{3885}{5157968}-\frac{23625 \sqrt{47}}{242424496}.
  \] 
  Moreover, $K_i$ has no dominant singularity other than its radius of convergence. 
  For $\nu\neq \tilde \nu_c$, the  behaviour of $K_i$ near $w=\tilde\rho_\nu$ is the standard singular behaviour of  planar maps series:
  \[
    K_i= \tilde\alpha_{i,\nu} +\tilde\beta_{i,\nu}(1-w/\tilde\rho_\nu)+\tilde\gamma_{i,\nu} (1-w/\tilde\rho_\nu)^{3/2}\,(1+o(1)),
  \] 
 where $\tilde\gamma_{i,\nu} \neq0$.   At $\nu=\tilde\nu_c$,  the nature of the singularity changes:
  \[
    K_i= \tilde\alpha_{i,\tilde\nu_c} +\tilde\beta_{i,\nu_c}(1-w/\tilde\rho_{i,\tilde\nu_c})+\tilde\gamma_{i,\tilde\nu_c} (1-w/\tilde\rho_{\tilde\nu_c})^{6/5}\,(1+o(1))
  \]
  where $\tilde\gamma_{i,\tilde\nu_c} \neq0$.
  In asymptotic terms:
\[ 
  [w^n]K_i\sim
  \left\{
    \begin{array}{ll}
      \tilde\kappa_{i,\nu}\, \left(\tilde\rho_\nu\right)^{-n}  n^{-5/2}&\hbox{ for  } \nu\not=\tilde\nu_c,
      \\
      \tilde \kappa_{i, \tilde\nu_c} \left(\tilde\rho_{\tilde\nu_c}\right)^{-n}  n^{-11/5}&\hbox{ for  } \nu=\tilde\nu_c.
    \end{array}
  \right.
  \] 
\end{Proposition}

\noindent{\bf Remark.}
The corresponding result for the Ising model on cubic maps, with critical value $1+2\sqrt 7$ and exponent $4/3$,  can be found in~\cite{Chen-Turunen-phase}; see also~\cite{BK87}.

\begin{proof}
{ \ }  
  
  \noindent  {\sc{\bf A. The case $\boldsymbol{i=1}$.}}
  We begin with the case $i=1$. By~\eqref{eq:KT} and Lemma~\ref{lem:Delta}, we first observe that, for $\nu\neq 1$ (a case that will be studied separately), any singularity of $K_1(\nu, \cdot)$ is a zero of $\tilde \Delta:=\tilde\Delta_0\tilde\Delta_1\tilde\Delta_2$, where   $\tilde\Delta_1$ and $\tilde \Delta_2$ are given in the proposition, and
  \[
    \tilde \Delta_0(\nu, w)= (\nu-1)^2\Delta_0 \left(\nu_*, \frac 1 3 (\nu-1)^3 w\right)
    =16 \left(\nu +2\right) \left(\nu -1\right)^{3} w -\left(\nu -4\right)^{2},
  \]
  with $\nu_*=(\nu+2)/(\nu-1)$.
  The real positive branches of these three polynomials are shown in Figure~\ref{fig:radius-cubic}.

 It also follows from~\eqref{eq:KT} that the series $\partial _w K_1:=\DK$ satisfies
  \[
    \DK(\nu,w)= \frac 1{\nu-1}\, \DT\left(\nu_*, \frac1 3 (\nu-1)^3 w\right).
  \]
 We then derive from~\eqref{alg:dT1} the minimal polynomial of $\DK$, again of degree~$11$ in $\DK$.

 \noindent     {\bf For $\boldsymbol{\nu=1}$}, the degree of $\DK$ drops to $3$ (we are then counting cubic maps with a weight~$3$ on each vertex), and the radius of $\DK$ is found to be $\tilde \rho_1=\sqrt{3}/324$, with a square root behaviour. We observe that $\tilde\rho_1$ is a root of $\tilde \Delta_2$. Moreover, the discriminant of the minimal polynomial of~$\DK$ has a second root, namely $-\rho_1$ (while the leading coefficient is constant). But this value is only singular for the other two solutions of the minimal polynomial of $\DK$. This proves all claims of the proposition for $\nu=1$.

\noindent   {\bf For $\boldsymbol{\nu>1}$}, we have $\nu_*>1$, and Eq.~\eqref{eq:KT} offers an extremely convenient shortcut to derive the results of Proposition~\ref{prop:asympt-cubic} from those of Proposition~\ref{prop:asympt}.  In particular, the radius of convergence~$\tilde \rho_\nu$ of $K_i(\nu,\cdot)$ is then $3 \rho_{\nu_*}/(\nu-1)^3$.  A transition will occur at the dual value $\tilde \nu_c:= (\nu_c+2)/(\nu_c-1)=1+\sqrt{47}$.
Since $\nu_*$ is a  non-increasing function of $\nu$ for $\nu>1$,
the radius of convergence of $K_i$ is a root of $\tilde \Delta_1$  above $\tilde\nu_c$ and a root of $\tilde \Delta_2$ between $1$ and $\tilde\nu_c$, with continuity at $1^ +$ as we have just seen. The nature and uniqueness of the dominant singularity transfer from $T_1$ to $K_1$.  So the proposition is proved for $i=1$ and $\nu\ge 1$.

\noindent  {\bf For $\boldsymbol{\nu\in (0,1]}$}, the radius $\tilde \rho_\nu$ is a continuous function of $\nu$, and we know that it is a root of~$\tilde \Delta_2$ when $\nu=1$. 
The branch of~$\tilde \Delta_2$ that contains $\tilde\rho_1$ decreases continuously between $0$ and $1$, and meets no other branch of $\tilde \Delta:= \tilde \Delta_0 \tilde \Delta_1  \tilde \Delta_2$ on $[0,1]$: hence it gives the value of
the
radius for $\nu\in (0,1]$. In fact, by considering the minimal polynomial of $\DK$ for $\nu=0$, we see that continuity holds on the closed interval $[0,1]$. This concludes the determination of the radius of $K_1$.

 It remains to study the nature of the singularity of $K_1$ near $\tilde \rho_\nu$ and to prove  uniqueness of this dominant singularity, for $\nu\in [0,1)$.

 \smallskip

\noindent  {\bf Nature of the singularity at $\boldsymbol{\tilde \rho_\nu}$ for $\boldsymbol{\nu\in [0,1)}$.}  We have proved that $\tilde \rho_\nu$ is a root of $\tilde \Delta_2$ on this interval. Note that $\nu_*=(\nu+2)/(\nu-1)$ is then in $(-\infty, -2]$. By~\eqref{eq:KT}, we see that $\DT(\nu_*, \cdot)$ has coefficients with alternating signs, and
  a negative dominant singularity at $s_{\nu_*}:=(\nu-1)^3\tilde \rho_\nu/3$.
  This singularity is a root of $\Delta_2(\nu_*,\cdot)$.
   Moreover, $\DK(\nu, \tilde\rho_\nu)$ is continuous in $\nu\in [0,1)$ by positivity of its coefficients, which implies that  $\DT(\nu_*, s_{\nu_*})$ is a continuous function of  $\nu_* \in (-\infty, -2]$.  We can now recycle the ingredients and  calculations done at the end of the  proof of Lemma~\ref{lem:sq-root} when $\rho$ was a root of $\Delta_2(\nu, \cdot)$. In particular, $\DT(\nu_*,s_{\nu_*})$ satisfies an equation of degree $9$ over $\qs(\nu_*)$, which is also satisfied by $\DT(\nu_*, \rho_{\nu_*})$ when $\nu_*>\nu_c$. 
Moreover,     upon replacing $\nu$ by $\nu_*$, 
the polynomial $\overline \Pol(Y, \eps)$ obtained  at the end of the  proof of Lemma~\ref{lem:sq-root} vanishes 
for   $Y=\DT(\nu_*,w)-\DT(\nu_*,s_{\nu_*})$ and $\eps=1- w/s_{\nu_*}$. Recall that all points $(i,j)$ in its support satisfy $i+2j\ge 18$. We happily check that the coefficients of $Y^{18}\vareps^0$ and of $Y^0\vareps^9 $ have no root in $(-\infty, -2]$, and conclude that the expansion of $\DT(\nu_*,\cdot)$ at its negative dominant singularity $s_{\nu_*}$, or equivalently of $K_1(\nu, \cdot)$ at its radius of convergence $\tilde\rho_\nu$, is of a square root type.
 
\smallskip
\noindent  {\bf Uniqueness of the dominant singularity for $\boldsymbol{\nu\in [0,1)}$.} This boils down to proving that $\DT(\nu_*, \cdot)$ has a unique dominant singularity  $s_{\nu_*}$ (which is then negative) for $\nu_*\le -2$. Imagine there is another one, say $s$. Then the equations of System~\eqref{syst} hold for some $j\in\{0,1,2\}$, with $i=2$, $\nu$ replaced by~$\nu_*$  and $\rho$ replaced by $s_{\nu_*}$. We have already computed with {\tt msolve} estimates of the solutions of these systems in the proof of Lemma~\ref{lem:unique}. Now the solutions we are interested in should satisfy, among other conditions, the following simple ones:
\[
  \nu_*\le -2, \qquad s_{\nu_*}<0.
\]
But one checks that there are no such solutions. Hence there is only one dominant singularity in $\DK(\nu, \cdot)$.

\medskip
\noindent {\sc{\bf B. The case $\boldsymbol{i>1}$.}}   Let us now consider the case $i>1$, with $\nu\ge 0$. The idea is to adapt the proof of Lemma~\ref{lem:Ti}, in a way that includes the case $\nu=0$. In other words, we should relate the series $K_i$ and $K_1$ by combinatorial constructions that do not create monochromatic edges. The key is the following bounds, explained by Figure~\ref{fig:KiK1}:
  \[
    [w^{n-2i+2}] K_1   \le  [w^n]K_i \le  [w^{n+2i-4}] K_1.
  \]
  Based on this, the rest of the proof mimics the proof of Lemma~\ref{lem:Ti}.
\end{proof}

  \begin{figure}[htb]
    \centering
    \scalebox{0.78}{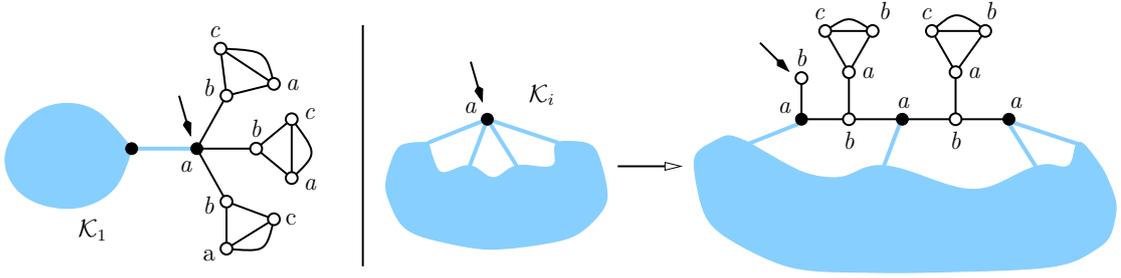}
    \caption{\emm Left., Construction of a map of $\mathcal K_i$ from a map of $\mathcal K_1$ (here, $i=4$).  \emm Right., Construction of a map of $\mathcal K_1$ from a map of $\mathcal K_i$ (with $i=4$ again). The set of colours is $\{a,b,c\}$. No monochromatic edge is created.}     
    \label{fig:KiK1}
  \end{figure}

 \subsection{Negative values of $\boldsymbol{\nu}$}

 It seems to be of interest, in the physics literature, to examine the position and nature of dominant singularities of the $3$-Potts series also when $\nu<0$. Indeed, this has been done on some regular lattices (see e.g.~\cite{
   jacobsen-triangle-potts}). We briefly return here to the case of near-triangulations, and present, without proof, what we predict in this case, based on the minimal polynomial~\eqref{alg:dT1} of $\DT$.  

Here are some of the results that we have rigorously established so far:
 \begin{itemize}
 \item for $\nu \in [0,\nu_c)$, the radius of convergence $\rho_\nu$ of $T_1$ (or more precisely, of $T_1/\nu$) is given by the branch of $\Delta_1$ that equals $1/8$ at $\nu=0$, and the series $T_1/\nu$ has a singularity in $(1-w/\rho_\nu)^{3/2}$ near its radius (Section~\ref{sec:asympt-triang}),
   \item for $\nu \le -2$, the (unique) dominant singularity $s_\nu$ of $T_1$ is given by the negative branch of $\Delta_2$ that tends to $0$ as $\nu$ tends to $-\infty$, say  $\mathcal B_2'$, and the series $T_1$ has a singularity in $(1-w/s_\nu)^{3/2}$ near $s_\nu$ (from Section~\ref{sec:asympt-cubic}).
   \end{itemize}
  \begin{figure}[htb]
  \centering
  \includegraphics[width=60mm]{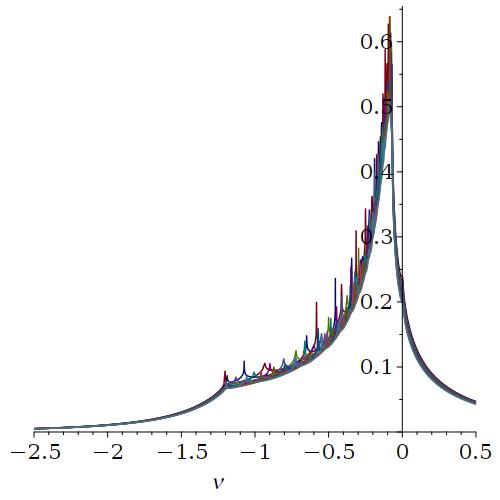}
  \hskip 10mm
    \includegraphics[width=60mm]{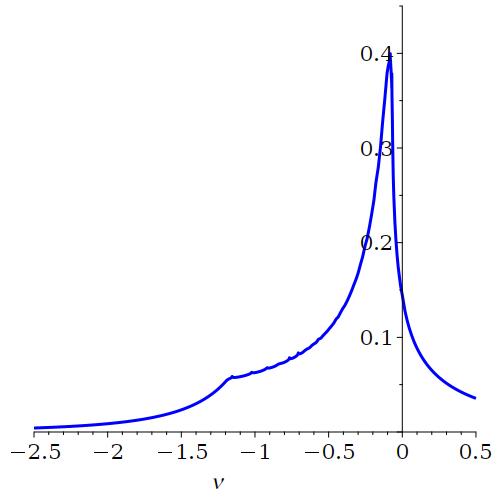}
 \caption{Estimates of the radius of convergence of $T_1$, as a function of $\nu$, by $\liminf \left|[w^n] T_1\right|^{-1/n}$. \emm Left,: for $n\in \llbracket 20, 30\rrbracket$ already, three phases seem to appear between $-2$ and $0$. \emm Right,: for $n=120$.}  
  \label{fig:radius_est}
\end{figure}

By combining exact computations for specific values of $\nu$, completed with estimates of the singularities and exponents obtained via differential approximants~\cite{DiffApprox}, and studies of the branches of $\Delta_0, \Delta_1$ and $\Delta_2$, here is what we predict for values of $\nu$ in $(-2,0)$, starting, say, from the largest of these values; see Figure~\ref{fig:radius_est} for a confirmation, and our \Maple\ session for details.
 
   \begin{itemize}
   \item {\bf Continuity at $\boldsymbol 0$}: the (positive) radius of convergence $\rho_\nu$ remains a dominant singularity, and lies on the branch $\mathcal B_1$ of $\Delta_1$ that goes through $(0,1/8)$ as long as $\nu>\nu_d:=1-12 /\sqrt{127}$ (Figure~\ref{fig:asympt-neg}, left). The singular exponent remains $3/2$.
   \item {\bf A negative critical value of $\boldsymbol \nu$:} at $\nu_d:=1-12 /\sqrt{127}\simeq -0.0648$, the branch $\mathcal B_1$ meets another branch of $\Delta_1$ at the point $(\nu_d,\rho_{\nu_d})$, with
     \[
       \rho_{\nu_d}=\frac 5{4 \cdot 17^3} \left( 3\cdot 5^2\cdot 7+\frac {13\cdot 43\sqrt{127}}{12}\right)\simeq 0.267.
     \]
     For $\nu=\nu_d$,  \emph{the nature of the singularity changes, with an exponent $4/3$} in $T_1$, which happens to be the exponent for the critical Ising model on planar maps (this part is rigorously proved).
   \item{\bf Two dominant singularities.} On the left of the point $(\nu_d,\rho_{\nu_d})$,  the two real branches of $\Delta_1$  become a pair of complex conjugate branches. They correspond to two dominant singularities $s_\nu$ and $\overline{s_\nu}$ of $T_1$, resulting in oscillating coefficients. The exponent remains~$1/2$. This can be checked rigorously for $\nu=-1$ for instance (Figure~\ref{fig:asympt-neg}, middle).
     \item{\bf Reaching the $\boldsymbol{\Delta_2}$-regime.} As $\nu$ approaches $\nu_e\simeq -1.1832$ from above, with $\nu_e$ an algebraic number of degree $435$, the modulus of the complex singularities $s_\nu$ and $\overline{s_\nu}$ becomes as large as the modulus of the point of the branch $\mathcal B_2'$ lying at abscissa $\nu$. Below this value, one dominant singularity of $T_1$ lies on  the negative branch $\mathcal B_2'$ of $\Delta_2$, as is the case for $\nu\le -2$ (Figure~\ref{fig:asympt-neg}, right). We do not know what the exponent is at $\nu_e$.
   \end{itemize}

   \begin{figure}[htb]
  \centering
  \includegraphics[width=44mm]{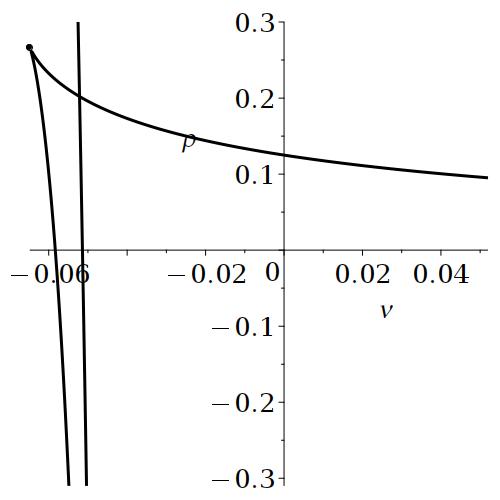} 
  \hskip 3mm \includegraphics[width=50mm]{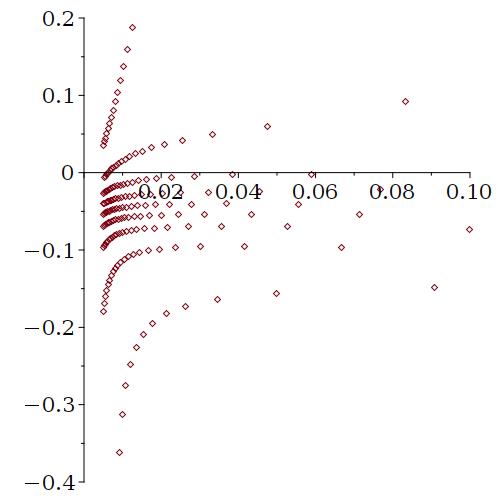}  \hskip 2mm
   \includegraphics[width=44mm]{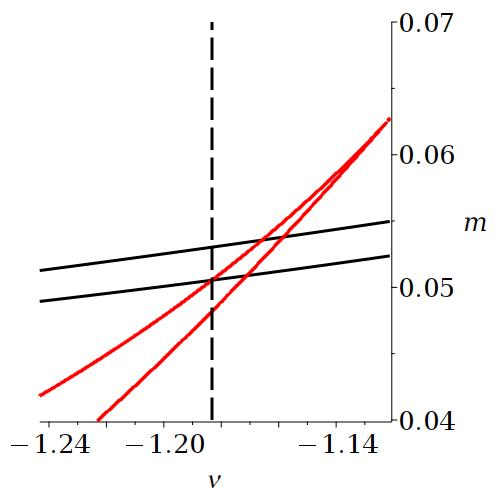} 
  \caption{\emm Left,. Above $\nu_d\simeq-0.0648$, the radius (and dominant singularity) remains on the branch of $\Delta_1$ that goes through $(0, 1/8)$. \emm Middle., For $\nu=-1$, the ratios between the coefficients of $w^n$ and $w^{n+1}$, plotted against $1/n$. \emm Right,. As~$\nu$ approaches $\nu_e\simeq-1.1832$ from above,  the modulus of the complex dominant singularities, which are roots of $\Delta_1$ (bottom black curve), becomes as large as the modulus of the (negative) branch $\mathcal B_2'$ (top red curve).} 
  \label{fig:asympt-neg}   
\end{figure}  
 
\section{Final comments}
 
\paragraph{\bf Exact results: future work.} We have thus determined, more than fifteen years after it was proved to be algebraic, the 3-Potts \gf\  of planar triangulations, or dually of planar cubic maps. Obviously, a more combinatorial proof would be highly desirable. In a forthcoming paper~\cite{BMN-26}, we address  the case of general planar maps. The starting point is a counterpart of the polynomial system of Proposition~\ref{prop:factor}. It is however bigger than for triangulations, and we resort to a different approach to solve it, based on evaluation-interpolation. Also, there are more parameters in this model: the number of edges is the natural size, but one can also record the number of vertices, and this gives rise to a richer singular landscape.

\bigskip
\paragraph{\bf Asymptotic results: some missing tools?} 
The singular analysis of the series $T_1$ (Section~\ref{sec:asympt-triang}) takes almost a third of the paper. The difficulty lies in the fact that $T_1$ is a series in $w$ with coefficients that depend (polynomially) on the parameter $\nu$. It is not too hard to determine the radius of convergence $\rho_\nu$ of $T_1$ (Lemma~\ref{lem:radius}), but the next two steps are more delicate, namely:
\begin{itemize}
\item determining the nature of the singularity of $T_1$ near $\rho_\nu$  (Lemma~\ref{lem:sq-root}),
  \item proving that $T_1$ has no other dominant singularity (Lemma~\ref{lem:unique}).
  \end{itemize}
  We observe that the singular exponent of $T_1$ at $\rho_\nu$ only changes at $\nu_c$, being constant on $(0, +\infty)\setminus\{\nu_c\}$. Could there be a general result that would guarantee, for instance, that $(0, +\infty)$ splits as a finite union of intervals on which the exponent is constant? This would allow one to consider only one value of $\nu$ per interval, which would be much easier. For instance, assume that $F$ is an algebraic series in $\ns[\nu][[w]]$, with minimal polynomial $\Pol(z) \in \qs[\nu,w][z]$.
  Assume that the product of the leading coefficient and the discriminant of $\Pol$ factors over $\qs[\nu,w]$ as
  \[
    D_0(\nu)^{e_0}     D_1(\nu,w)^{e_1} \cdots     D_k(\nu,w)^{e_k} ,
  \]
with the $D_i$'s distinct and irreducible. Consider a value $\nu$ such that $D_0(\nu)\neq 0$, and assume that the radius $\rho_\nu$ lies on a branch of $D_i$: if $\partial_w D_i(\nu, \rho_\nu)\neq 0$ and $D_j(\nu, \rho_\nu)\neq0$, can we say that the critical exponent is constant in a neighbourhood of $\nu$? That is, is the exponent constant except at the finitely many points where ``something special'' happens in the singularity landscape? 
  
Regarding uniqueness of the dominant singularity, we believe to have introduced new useful tools in the first proof of  Lemma~\ref{lem:unique} by forming a number of polynomial conditions. It is worth noting that, in a personal communication, Linxiao Chen suggested an additional, non-obvious polynomial relation, inspired from   an unpublished preprint~\cite[Prop.~26]{Chen-trees}. This remains to be explored. {Moreover, the criterion of Proposition~\ref{lem:compo_unique}, suggested by discussions with Andrew Elvey Price, should also be applicable to other \gfs.}

\bigskip
\noindent {\bf Acknowledgements.} We are grateful to several colleagues for very informative discussions during the preparation of this paper. In particular, we thank Marie Albenque for live explanations on her paper~\cite{albenque-Ising} with Laurent M\'enard and Gilles Schaeffer; Alin Bostan, for providing the reference~\cite{Braverman}; Linxiao Chen for e-mail discussions on uniqueness of the dominant singularity;  Andrew Elvey Price for suggesting a combinatorial approach to the uniqueness of the dominant singularity; Mohab Safey El Din for his repeated help with {\tt msolve}, and Tony Guttmann for his help with differential approximants; and  Jesper Jacobsen and Andrea Sportiello for fascinating discussions on the physical (and non-physical) aspects of the Potts model. Finally, we thank our reviewer for his/her thorough report and useful suggestions.

\appendix

\section{The 1-catalytic equation for $\boldsymbol{T(y)}$}\label{app:1cat}
Here is the equation in one catalytic variable (namely, $y$), satisfied by $T(y)$, derived from Proposition~\ref{prop:eqinv} by expansion near $y=0$. It involves the four series $T_1, T_3, T_5$ and $T_7$, where $T_i=[y^i] T(y)$. Below, we write $\beta:=\nu-1$:
\begin{small}
  \allowdisplaybreaks
 \begin{multline} \label{1cat-complete}
       2916 \nu^{5} y^{12} \boldsymbol{T(y)}^{5}
    +27  \nu^{3} y^{9} \Big(
    \beta \left(37 \nu +17\right) y^{2}-36 \nu  \left(3 \nu +1\right) y +144 \nu^{2}\Big)
    \boldsymbol{T(y)}^{4}\\
    -2\nu y^6
    \Big(486 \boldsymbol{T_1} \,\nu^{4} y^{4}-81 \nu^{3} \left(5 \nu +1\right) w \,y^{4}+486 \nu^{4} w \,y^{3}-\left(56 \nu^{2}+59 \nu +2\right) \beta^{2} y^{4}   \\
    +9 \nu  \beta \left(38 \nu^{2}+40 \nu +3\right) y^{3}-9 \nu^{2} \left(116 \nu^{2}+11 \nu -19\right) y^{2}+486 \nu^{3} \left(3 \nu +1\right) y -972 \nu^{4}
    \Big)
    \boldsymbol{T(y)}^{3}\\
    -\nu y^3 \Big( 18 \nu^{2} \left(25 \nu^{2}-23 \nu -20\right) \boldsymbol{T_1} \,y^{6}-972 \boldsymbol{T_1} \,\nu^{4} y^{5}+972 \boldsymbol{T_1} \,\nu^{4} y^{4}+324 \boldsymbol{T_3} \,\nu^{4} y^{6}+972 \nu^{4} w^{2} y^{6} 
    \\
    -6 \nu  \beta \left(37 \nu^{2}+40 \nu +4\right) w \,y^{6}+54 \nu^{2} \left(17 \nu^{2}+3 \nu -2\right) w \,y^{5}-216 \nu^{3} \left(7 \nu +2\right) w \,y^{4}+972 \nu^{4} w \,y^{3} \\
    -4 \left(\nu +2\right) \beta^{3} y^{6}+2 \left(20 \nu^{2}+53 \nu +5\right) \beta^{2} y^{5}-\beta \left(226 \nu^{3}+276 \nu^{2}-15 \nu -1\right) y^{4}\\
    +6 \nu  \left(106 \nu^{3}+35 \nu^{2}-64 \nu -5\right) y^{3}-3 \nu^{2} \left(353 \nu^{2}+104 \nu -25\right) y^{2}+324 \nu^{3} \left(3 \nu +1\right) y -432 \nu^{4}
    \Big) \boldsymbol{T(y)}^{2}\\
    -2\nu\Big(
    324 \nu^{4} w^{2} y^{6}+108 \boldsymbol{T_3} \,\nu^{4} y^{6}-324 \boldsymbol{T_1} \,\nu^{4} y^{5}+2 \beta^{3} y^{7}+162 \boldsymbol{T_1}^{2} \nu^{4} y^{8}+54 \boldsymbol{T_5} \,\nu^{4} y^{8}-108 \boldsymbol{T_3} \,\nu^{4} y^{7}\\
    -\left(2 \nu^{2}+10 \nu +1\right) \beta^{2} y^{6}-\nu  \left(37 \nu^{3}+22 \nu^{2}-38 \nu -3\right) y^{4}-\nu^{2} \left(79 \nu^{2}+31 \nu -2\right) y^{2}+27 \nu^{3} \left(5 \nu +19\right) \boldsymbol{T_1} w \,y^{8}\\
    +2 \nu  \beta \left(34 \nu^{2}+46 \nu +1\right) w \,y^{7}-\nu  \left(17 \nu -11\right) \left(13 \nu^{2}+13 \nu +1\right) w \,y^{6}+162 \boldsymbol{T_1} \,\nu^{4} y^{4}+162 \nu^{4} w \,y^{3}\\
    +\left(36 \nu^{4}-69 \nu^{3}+9 \nu^{2}+126 \nu +6\right) \boldsymbol{T_1} \,y^{8}-6 \nu^{2} \left(25 \nu^{2}-23 \nu -20\right) \boldsymbol{T_1} \,y^{7}+6 \nu^{2} \left(52 \nu^{2}-23 \nu -20\right) \boldsymbol{T_1} \,y^{6}\\
    +3 \nu^{2} \left(7 \nu^{2}-41 \nu -20\right) \boldsymbol{T_3} \,y^{8}+9 \nu^{2} \left(13 \nu^{2}-32 \nu -17\right) w^{2} y^{8}-27 \nu^{3} \left(11 \nu +1\right) w^{2} y^{7}\\
    -\left(8 \nu^{2}+28 \nu +3\right) \beta^{2} w \,y^{8}+6 \nu^{2} \left(65 \nu^{2}+20 \nu -4\right) w \,y^{5}-9 \nu^{3} \left(41 \nu +13\right) w \,y^{4}\\
    +\beta \left(12 \nu^{3}+26 \nu^{2}-10 \nu -1\right) y^{5}+\nu  \left(68 \nu^{3}+33 \nu^{2}-27 \nu -2\right) y^{3}+18 \nu^{3} \left(3 \nu +1\right) y -18 \nu^{4}
    \Big)\boldsymbol{T(y)}\\
    + \Big(
    972 \boldsymbol{T_1} \,\nu^{5} w^{2} y^{7}-324 \boldsymbol{T_1} \boldsymbol{T_3} \,\nu^{5} y^{7}+108 \boldsymbol{T_1} \,\nu^{5} y^{2}-36 \boldsymbol{T_1} \,\nu^{5} y -36 \nu^{5} w +108 \boldsymbol{T_1}^{2} \nu^{5} y^{6}-36 \boldsymbol{T_7} \,\nu^{5} y^{7}\\
    -108 \boldsymbol{T_1}^{2} \nu^{5} y^{5}+36 \boldsymbol{T_5} \,\nu^{5} y^{6}-36 \boldsymbol{T_5} \,\nu^{5} y^{5}+72 \boldsymbol{T_3} \,\nu^{5} y^{4}-108 \nu^{5} w^{2} y^{3}-36 \boldsymbol{T_3} \,\nu^{5} y^{3}\\
    +2 \nu^{2} \left(73 \nu^{3}+339 \nu^{2}+768 \nu +170\right) \boldsymbol{T_1} w \,y^{7}+18 \nu^{4} \left(5 \nu +19\right) \boldsymbol{T_1} w \,y^{6}-18 \nu^{4} \left(5 \nu +19\right) \boldsymbol{T_1} \,y^{5} w \\
    +4 \nu^{3} \left(2 \nu +1\right) \left(17 \nu -20\right) \boldsymbol{T_1} \,y^{4}-18 \nu^{4} \left(29 \nu +25\right) \boldsymbol{T_3} w \,y^{7}+3 \nu^{3} \left(71 \nu^{2}+92 \nu -55\right) \boldsymbol{T_1}^{2} y^{7}\\
    +\left(-4 \nu^{5}+12 \nu^{4}+34 \nu^{3}+98 \nu^{2}+142 \nu +6\right) \boldsymbol{T_1} \,y^{7}+2 \nu  \left(12 \nu^{4}-23 \nu^{3}+3 \nu^{2}+42 \nu +2\right) \boldsymbol{T_1} \,y^{6}\\
    -2 \nu  \left(37 \nu^{4}-46 \nu^{3}-17 \nu^{2}+42 \nu +2\right) \boldsymbol{T_1} \,y^{5}-2 \nu^{3} \left(79 \nu^{2}-23 \nu -20\right) \boldsymbol{T_1}   \,y^{3}\\
    -2 \nu  \left(23 \nu^{4}+47 \nu^{3}+100 \nu^{2}+62 \nu +2\right) \boldsymbol{T_3} \,y^{7}+2 \nu^{3} \left(7 \nu^{2}-41 \nu -20\right) \boldsymbol{T_3} \,y^{6}-2 \nu^{3} \left(25 \nu^{2}-41 \nu -20\right) \boldsymbol{T_3}  \,y^{5}\\
    +2 \nu^{3} \left(29 \nu^{2}+59 \nu +20\right) \boldsymbol{T_5}  \,y^{7}-54 \nu^{4} \left(13 \nu +17\right) w^{3} y^{7}-\nu  \left(62 \nu^{4}+132 \nu^{3}+387 \nu^{2}+274 \nu +9\right) w^{2} y^{7}\\
    -9 \nu^{3} \left(19 \nu^{2}-20 \nu -11\right) y^{5} w^{2}+18 \nu^{4} \left(11 \nu +1\right) w^{2} y^{4}+4 \nu  \beta^{3} w \,y^{7}-2 \nu^{2} \left(37 \nu^{3}+22 \nu^{2}-38 \nu -3\right) w \,y^{4}\\
    +2 \nu^{2} \left(68 \nu^{3}+33 \nu^{2}-27 \nu -2\right) w \,y^{3}-2 \nu^{3} \left(79 \nu^{2}+31 \nu -2\right) w \,y^{2}
    +36 \nu^{4} \left(3 \nu +1\right) w y \\
    +36 \nu^{3} \left(2 \nu +1\right) \left(\nu -3\right) w^{2} y^{6}-2 \nu  \left(2 \nu^{2}+10 \nu +1\right) \beta^{2} w \,y^{6}+2 \nu  \beta \left(12 \nu^{3}+26 \nu^{2}-10 \nu -1\right) w \,y^{5}
    \Big)=0  .
  \end{multline}
\end{small}

\section{Explicit minimal polynomials of the radii of convergence}

\subsection{Triangulations}\label{sec:app-triang}
Depending on the value of $\nu\ge0$, the minimal polynomial of $\rho_\nu$ (the radius of $T_i(\nu, \cdot)$) is one of the following two polynomials, where we denote $\beta=\nu-1$ (see Proposition~\ref{prop:asympt}):
\begin{small}  \allowdisplaybreaks
  \begin{multline} \label{Delta1}
    \Delta_1(\nu,\rho)=  5135673858195456 \nu^{12} \rho^{5}+3869835264 \nu^{9} \beta  \left(23311 \beta^{2}-94464\right) \rho^{4}\\
    -221184 \nu^{6} \left(13675471 \beta^{6}-31778190 \beta^{4}-2827548 \beta^{2}+5971968\right) \rho^{3}\\
    -6912 \nu^{3} \beta  \left(16987825 \beta^{8}-94263704 \beta^{6}+178105122 \beta^{4}-125057520 \beta^{2}+22523184\right) \rho^{2}\\
    -32 \beta^{2}\left(44998721 \beta^{10}-354609900 \beta^{8}+1100056473 \beta^{6}-1675428138 \beta^{4}+1251637596 \beta^{2}-366996096\right)  \rho\\
    -\beta (\beta-1)^{3}  \left(6509057 \beta^{8}-44137521 \beta^{6}+110675808 \beta^{4}-118226304 \beta^{2}+45349632\right) ,
  \end{multline}
  \begin{multline}\label{Delta2}  \allowdisplaybreaks
    \Delta_2(\nu,\rho)=173133498956120064000 \nu^{23} \rho^{9}-86566749478060032 \nu^{20} \beta  \left(2 \beta^{2}-525\right) \rho^{8}\\
    +2348273369088 \nu^{17} \left(30157 \beta^{4}+60528 \beta^{2}+1674720\right) \beta^{2} \rho^{7}\\
    -764411904 \nu^{14} \beta  \left(554491 \beta^{8}-37429344 \beta^{6}+63722112 \beta^{4}-384466944 \beta^{2}+298598400\right) \rho^{6}\\
    -42467328 \nu^{11} \beta^{2} \left(1797719 \beta^{10}-24965736 \beta^{8}+79207569 \beta^{6}-298516608 \beta^{4}+421311240 \beta^{2}-41990400\right)  \rho^{5}\\
    -884736 \nu^{8}\beta^{3} \left(750262 \beta^{12}-5588808 \beta^{10}+11945370 \beta^{8}-215187504 \beta^{6}+744324669 \beta^{4}-1061891640 \beta^{2} \right.
    \\
    \left. +1062357120\right)  \rho^{4}
    -2048 \nu^{5}\beta^{2} \left(6476656 \beta^{16}-146287584 \beta^{14}+1597400568 \beta^{12}-9587073168 \beta^{10} \right.
    \\
    \left. +29838912921 \beta^{8}-51034914594 \beta^{6}+52101038457 \beta^{4}-21850334496 \beta^{2}-7652750400\right)  \rho^{3}\\
    -768 \nu^{2}\beta^{3}  \left(286992 \beta^{18}-7681792 \beta^{16}+88324784 \beta^{14}-561120928 \beta^{12}+2166482658 \beta^{10}-5422315320 \beta^{8} \right.
    \\
    \left. +9146245860 \beta^{6}-9455624856 \beta^{4}+4707143523 \beta^{2}-637729200\right) \rho^{2}\\
    -96 \nu  \beta^{4}  \left(\beta -1\right)^{2} \left(127552 \beta^{12}-3288320 \beta^{10}+34397596 \beta^{8}-192023736 \beta^{6}+599850702 \beta^{4} \right.
    \\
    \left. -880127532 \beta^{2}+449067645\right)\rho\\
    +\beta^{3} \left(\beta-1 \right)^{4} \left(510208 \beta^{10}-11133824 \beta^{8}+95834752 \beta^{6}-383558976 \beta^{4}+649066608 \beta^{2}-358722675\right) .
  \end{multline}
\end{small}

\subsection{Cubic maps}\label{sec:app-cubic}
Depending on the value of $\nu\ge0$, the minimal polynomial of $\tilde\rho_\nu$ (the radius of $K_1(\nu, \cdot)$) is one of the following two polynomials, where we write $\beta=\nu-1$ (see Proposition~\ref{prop:asympt-cubic}):
\begin{small}
  \allowdisplaybreaks
  \begin{multline} \label{DeltaC1}
    \tilde \Delta_1(\nu,\rho)= 28991029248 \left(3+\beta \right)^{12} \beta^{15} \rho^{5}-1769472 \left(10496 \beta^{2}-23311\right) \left(3+\beta \right)^{9} \beta^{12} \rho^{4}\\
    -8192 \left(8192 \beta^{6}-34908 \beta^{4}-3530910 \beta^{2}+13675471\right) \left(3+\beta \right)^{6} \beta^{9} \rho^{3}\\
    -2304 \left(30896 \beta^{8}-1543920 \beta^{6}+19789458 \beta^{4}-94263704 \beta^{2}+152890425\right) \left(3+\beta \right)^{3} \beta^{6} \rho^{2}\\
    +864 \left(55936 \beta^{10}-1716924 \beta^{8}+20684298 \beta^{6}-122228497 \beta^{4}+354609900 \beta^{2}-404988489\right) \beta^{3} \rho \\
    +27 \left(6912 \beta^{8}-162176 \beta^{6}+1366368 \beta^{4}-4904169 \beta^{2}+6509057\right) \left(-3+\beta \right)^{3},
  \end{multline}

  \begin{multline}\label{DeltaC2}  
    \tilde \Delta_2(\nu,\rho)=8796093022208000 \beta^{21} \left(3+\beta \right)^{23} \rho^{9}+118747255799808 \left(175 \beta^{2}-6\right) \beta^{18} \left(3+\beta \right)^{20} \rho^{8}\\
    +86973087744 \left(186080 \beta^{4}+60528 \beta^{2}+271413\right) \beta^{15} \left(3+\beta \right)^{17} \rho^{7}\\
    -84934656 \left(11059200 \beta^{8}-128155648 \beta^{6}+191166336 \beta^{4}-1010592288 \beta^{2}+134741313\right)\times\\
    \beta^{12} \left(3+\beta \right)^{14} \rho^{6}
    +382205952 \left(172800 \beta^{10}-15604120 \beta^{8}+99505536 \beta^{6}-237622707 \beta^{4}+674074872 \beta^{2}\right.
    \\
    \left. -436845717\right) \beta^{9} \left(3+\beta \right)^{11} \rho^{5}-71663616 \left(4371840 \beta^{12}-39329320 \beta^{10}+248108223 \beta^{8}-645562512 \beta^{6}\right.
    \\
    \left.+322524990 \beta^{4}-1358080344 \beta^{2}  +1640822994\right) \beta^{6} \left(3+\beta \right)^{8} \rho^{4}+4478976 \left(1166400 \beta^{16}+29973024 \beta^{14}\right.
    \\
    \left.-643222697 \beta^{12}+5670546066 \beta^{10}-29838912921 \beta^{8}+86283658512 \beta^{6}-129389446008 \beta^{4}\right.
    \\
    \left.+106643648736 \beta^{2}-42493340016\right) \beta^{3} \left(3+\beta \right)^{5} \rho^{3}+5038848 \left(291600 \beta^{18}-19370961 \beta^{16}+350208328 \beta^{14}\right.
    \\
    \left.-3048748620 \beta^{12}+16266945960 \beta^{10}-58495031766 \beta^{8}+136352385504 \beta^{6}-193166302608 \beta^{4}\right.
    \\
    \left.+151200711936 \beta^{2}-50839771824\right) \left(3+\beta \right)^{2} \rho^{2}-1889568 \beta  \left(3+\beta \right) \left(616005 \beta^{12}-10865772 \beta^{10}\right.
    \\
    \left.+66650078 \beta^{8}-192023736 \beta^{6}+309578364 \beta^{4}-266353920 \beta^{2}+92985408\right) \left(-3+\beta \right)^{2} \rho \\-19683 \left(492075 \beta^{10}-8013168 \beta^{8}+42617664 \beta^{6}-95834752 \beta^{4}+100204416 \beta^{2}-41326848\right) \left(-3+\beta \right)^{4}
    .
  \end{multline}
\end{small}

\section{Proof of Proposition~\ref{lem:compo_unique}}
\label{app:unique-compo}

\begin{proof}
  First, observe that the conditions $g_{0,0}=0$ and $g_{0,1}<1$ imply that $U(z)=\sum _{n\ge 1} u_n z^n$ is completely characterized by~\eqref{eq:UG},
  once $G$ is fixed: the coefficient  $u_n$ is determined uniquely by induction on $n\ge 1$ (recall that $u_0=0$ by assumption). Upon replacing $U$ by $(1-g_{0,1})U$, we could even assume that $g_{0,1}=0$. Note also that $U$ has non-negative coefficients since $G$ has non-negative coefficients and  $g_{0,1}<1$.
  
  Let us  prove that the radius of convergence $\rho$ of $U$ is finite. Let $m\ge 0, n\ge 2$ such that $g_{m,n}>0$, and let $k\ge 0$ such that $g_{k,0} >0$: such a $k$ exists, otherwise $U(z)$ would be uniformly~$0$, and in particular periodic. Let $V\equiv V(z)$ be the unique power series satisfying
  \[
    V=
    g_{k,0} z^k + g_{m,n} z^m V^n.
  \]
  Then $V$ is an algebraic series, which has a finite radius of convergence. Moreover, its coefficients bound from below the coefficients of $U$. Hence $U$ itself has a finite radius of convergence $\rho$.  By Pringsheim's theorem, $\rho$ is a singularity of $U$.

  Let us prove that $\rho>0$. Recall that $G(z,u)$ is analytic around $(0,0)$, that $G'_u(0,0)=g_{0,1}<1$ and $U(0)=0$. Then the  analytic version of the implicit function theorem implies that $U$ is analytic around $0$.

  Let us now prove that $U(z)$ converges at $z=\rho$. We first note that $U$ is an increasing continuous function of $z$ on $[0, \rho)$, with $U(0)=0$. Moreover, there exist $m\ge 0$ and $n\ge 2$ such that on the interval $(0, \rho)$, we have  $U(z) \ge  g_{m,n} z^m U(z)^n$, where $ g_{m,n}>0$. In particular, for $z>\rho/2$, we have $1\ge g_{m,n} (\rho/2)^m U(z)^{n-1}$, hence $U(z)$ remains bounded on $[0, \rho)$, and thus converges at $\rho$ (by non-negativity of its coefficients).

  \smallskip
  We now consider the curve of $\rs_{\ge0}^2$ consisting of the points $(z, U(z))$  for $0\le z \le \rho$. It starts at $(0,0)$ in $\mathcal R$, and then increases continuously to the point $(\rho, U(\rho))$. We study three cases, depending on the relative positions of this point and $\mathcal R$. Note that if $(z,u)\in \mathcal R$, then $[0,z] \times [0,u] \subset \mathcal R$. Also, $G(z,u)$ is analytic at any point of $\mathcal R^\circ$. Moreover, if $z_0>0$ and $(z_0, \vareps)\in \mathcal R^\circ$ for some $\vareps>0$, then $G(z,u)$ is analytic at $(z_0,u)$ as well.

  {\bf Case 1:} if $(\rho, U(\rho)) \in \mathcal R^\circ$, then it must be that $G'_u(\rho, U(\rho))=1$, otherwise the analytic version of the implicit function theorem would provide an analytic continuation of $U$ at $\rho$.  By Theorem VII.3 of~\cite{flajolet-sedgewick}, the series $U$ has a square root singularity at $\rho$, and $\rho$ is the unique dominant singularity.

  {\bf Case 2:} if $(\rho, U(\rho)) \in \overline{\mathcal R}\setminus \mathcal R^\circ$, then it must be that $G'_u(z, U(z))<1$ for $z\in [0, \rho)$. Indeed, we have
  $G'_u(0, U(0))=G'_u(0,0)=g_{0,1}<1$, and if the above inequality did not persist on $[0, \rho)$,  Theorem~VII.3 of~\cite{flajolet-sedgewick} would imply that $U$ has radius of convergence less than $\rho$. By continuity, and non-negativity of the coefficients of $G$, we have $G'_u(\rho, U(\rho))\le 1$.

  Now let $s\neq \rho$ have modulus $\rho$. We will prove that $U$ has an analytic continuation at $s$. The aperiodicity assumption implies that $|U(s)|< U(\rho)$. By assumption on the boundary of $\mathcal R$, this means that $(|s|, |U(s)|)=(\rho, |U(s)|)$ lies in $\mathcal R^\circ \cup \{(\rho,0)\}$. Hence $G(z,u)$ is analytic in a neighbourhood of $(|s|, |U(s)|)$, and hence, by non-negativity, in a neighbourhood of $(s, U(s))$.  Moreover,
  $|G'_u(s, U(s)) | \le G'_u(\rho, |U(s)|) < |G'_u(\rho, U(\rho)|\le 1$ (the strict inequality follows from the fact that $|U(s)|<U(\rho)$ and that $G'_u(z,u)$ actually depends on $u$). Hence, by the implicit function theorem, $U$ has an analytic continuation at the point $s$, which cannot be a singularity.

  \allowdisplaybreaks
  
  {\bf Case 3:}  finally, it is impossible to have  $(\rho, U(\rho)) \not \in \overline{\mathcal R}$, or even  $(\rho, U(\rho)) \not \in {\mathcal R}$, because  the series~$G$ \emm does, converge at $(\rho, U(\rho))$. Indeed, since $U$ converges at $\rho$ and all coefficients are non-negative, we have
  \begin{align*}
    U(\rho)= \sum_{n\ge 1 } u_n \rho^n
    &=
    \sum_{n\ge 1} \rho^n \left(\sum_{i,j \ge 0} g_{i,j} [z^{n-i}] U(z)^j \right) \quad \quad \text{by } \eqref{eq:UG}
    \\
    &=\sum_{i,j \ge 0} g_{i,j}  \rho^i\left( \sum_{n\ge i} \rho^{n-i}  [z^{n-i}] U(z)^j\right)
    \\
    &=\sum_{i,j \ge 0} g_{i,j}  \rho^i U(\rho)^j \qquad \text{since $U$ converges at } \rho.
  \end{align*}
  This concludes the proof of the proposition.
\end{proof}

\bibliographystyle{abbrv} 
\bibliography{three-coloured.bib}

\end{document}